\documentclass[11pt, reqno]{amsart}

\usepackage{graphicx,amsmath,amssymb,mathrsfs}
\usepackage{amsfonts,amsthm,epsfig,epstopdf}
\usepackage[mathscr]{euscript}
\usepackage[latin1]{inputenc}
\usepackage[numbers]{natbib}
\usepackage[usenames, dvipsnames]{color}

\newtheorem{theorem}{Theorem}[section]
\newtheorem{lemma}[theorem]{Lemma}
\newtheorem{proposition}[theorem]{Proposition}

\newtheorem{remark}[theorem]{Remark}

\def\H{\mathcal H}

\def\R{\mathbb R}
\def\N{\mathbb N}

\def\dist{\hbox{dist}}
\def\e{\varepsilon}
\def\s{\sigma}
\def\S{\Sigma}
\def\vphi{\varphi}
\def\Div{{\rm div}\,}

\def\l{\lambda}
\def\g{\gamma}
\def\k{\kappa}
\def\Om{\Omega}
\def\de{\delta}
\def\Id{{\rm Id}}

\def\spt{{\rm spt}}

\def\pa{\partial}
\def\trace{{\rm tr}}

\def\E{\mathcal{E}}

\def\F{\mathcal{F}}

\newcommand{\hd}{\mathrm{hd}}

\renewcommand{\>}{\rangle}
\renewcommand{\a}{\alpha}
\renewcommand{\b}{\beta}

\newcommand{\Lip}{{\rm Lip}}

\renewcommand{\Div}{{\rm div \,}}
\newcommand{\ov}{\overline}

\newcommand{\diam}{\mathrm{diam}}
\newcommand{\cc}{\subset\subset}

\newcommand{\C}{{\mathscr{C}}}

\def\H{\mathcal H}

\def\de{\delta}

\def\l{{\lambda}}
\newcommand{\vertiii}[1]{{\left\vert\kern-0.25ex\left\vert\kern-0.25ex\left\vert #1
    \right\vert\kern-0.25ex\right\vert\kern-0.25ex\right\vert}}
\def\S{\mathbb{S}}

\def\trace{{\rm tr}}

\newcommand{\DIVV}{{\rm div}}

\newcommand{\na}{{\nabla}}
\def\>{{\rangle}}

\numberwithin{equation}{section}

\title[$L^2$-bubbling into Wulff shapes]
{Bubbling
with $L^2$-almost constant mean curvature
\\and an Alexandrov-type theorem for crystals}

\author[Delgadino]{M. G. Delgadino}
\address[Matias Delgadino]{
\newline \indent Abdus Salam International Center for Theoretical Physics,
\newline \indent Strada Costiera 11, I-34151, Trieste, Italy.}
\email{mdelgadi@ictp.it}

\author[Maggi]{F. Maggi}
\address[Francesco Maggi]{
\newline \indent Abdus Salam International Center for Theoretical Physics,
\newline \indent Strada Costiera 11, I-34151, Trieste, Italy.
\newline \indent On leave from the University of Texas at Austin}
\email{fmaggi@ictp.it}

\author[Mihaila]{C. Mihaila}
\address[Cornelia Mihaila]{ \newline
\indent  Department of Mathematics, The University of Texas at Austin,  \newline
\indent 2515 Speedway Stop C1200, Austin, TX 78712, USA}
\email{cmihaila@math.utexas.edu}

\author[Neumayer]{R. Neumayer}
\address[Robin Neumayer]{ \newline
\indent  Department of Mathematics, The University of Texas at Austin,  \newline
\indent 2515 Speedway Stop C1200, Austin, TX 78712, USA}
\email{rneumayer@math.utexas.edu}

\begin{document}

\maketitle

\begin{abstract}
  {\rm A compactness theorem for volume-constrained almost-critical points of elliptic integrands is proven. The result is new even for the area functional, as almost-criticality is measured in an integral rather than in a uniform sense. Two main applications of the compactness theorem are discussed. First, we obtain a description of critical points/local minimizers of elliptic energies interacting with a confinement potential. Second, we prove an Alexandrov-type theorem for crystalline isoperimetric problems.}
\end{abstract}

\section{Introduction} \subsection{Overview}\label{section overview} The study of critical points in geometric variational problems often calls for the understanding of bubbling/concentration phenomena. Classical examples are discussed in the seminal papers of Brezis-Coron \cite{breziscoron84} and Struwe \cite{struwe84}, where the authors investigate immersed disks with almost-constant mean curvature and conformally flat metrics with almost-constant scalar curvature. As illustrated by the monographs \cite{struweBOOK,hebeyBOOK}, this kind of result plays an important role in various contexts.

Here we are interested in sets with almost-constant mean curvature, that is to say, in sets that are close to being critical in isoperimetric problems. Such sets arise in various contexts of physical and geometric importance, like capillarity theory and mean curvature flows. Depending on the application one has in mind, different ways of measuring almost-criticality are appropriate. For example, in the study of capillarity problems, one is naturally led to consider surfaces whose mean curvature is uniformly close to a constant. In geometric applications, we have a more complicated situation, as uniform proximity to constant mean curvature should be replaced by $L^2$-proximity.

Our starting point is the paper \cite{ciraolomaggi2017}, where, as illustrated in more detail below (Section \ref{section ciraolomaggi}), Ciraolo and the second author obtained a compactness result for boundaries whose mean curvature is uniformly close to a constant. In our main result, Theorem \ref{thm main1} below, we obtain two critical improvements of the compactness theorem from \cite{ciraolomaggi2017}, which require a substantial rethinking of many technical aspects of the original argument.

A first improvement consists of replacing uniform proximity with $L^2$-proximity. The main difficulty here is of course that uniform proximity to constant mean curvature, unlike $L^2$-proximity, carries information on the size of the mean curvature oscillation at {\it every} boundary point, and thus allows one to exploit powerful sliding/maximum principle arguments.

A second major improvement consists of replacing the area functional with a surface energy for a generic elliptic integrand. Elliptic integrands model anisotropic surface tensions and are thus of importance in numerous applications. From the mathematical viewpoint, the area functional is quite exceptional among elliptic integrands, and there are various steps in the argument of \cite{ciraolomaggi2017} where this fact was exploited.

Referring to Section \ref{section elliptic compactness} for more comments on the proof of Theorem \ref{thm main1}, we now discuss some applications of physical and geometric interest.

We start with Theorem \ref{thm locmin}, where we obtain a description of critical points and local minimizers in elliptic capillarity problems. This theorem is stated in Section \ref{section local min}, where we also provide additional context on this type of problem.

In Section \ref{section weak alexandrov}, we state Theorem \ref{thm main}, an extension of Theorem \ref{thm main1} to sequences of almost-critical points corresponding to elliptic integrands with degenerating ellipticity. The latter property allows these elliptic integrands to converge to an arbitrary (i.e., possibly non-smooth and non-elliptic) convex integrand, and our result proves convergence of almost-critical points (with sufficiently fast convergence of the first variation to a constant) to (possibly multiple copies of) the Wulff shape of the limit integrand. We propose an interpretation of this result as a suitable formulation of Alexandrov's theorem for generic anisotropic energies. Let us recall that for smooth, elliptic anisotropic energies one has a pointwise notion of mean curvature for which an exact analog of the classical Alexandrov's theorem holds \cite{HeLiMaGe09}. The situation is quite different for generic anisotropic energies, like crystalline energies, as in those cases the first variation of the energy does not even define a linear functional on the space of variations. In the proposed interpretation, we circumvent these difficulties by defining critical points of generic anisotropic problems as the accumulation points of almost-critical points of smooth elliptic anisotropic problems.

\subsection{Compactness in the Euclidean case}\label{section ciraolomaggi} We start by recalling the situation in the basic Euclidean case. The starting point is {\it Alexandrov's theorem}: if $\Om$ is a smooth bounded connected open set with constant mean curvature, then $\Om$ is a ball of radius $(n+1)|\Om|/P(\Om)=n/H_\Om$. Here, $|\Om|$ and $P(\Om)=\H^n(\pa\Om)$ are the volume and the perimeter of $\Om$, while $H_\Om$ is the scalar mean curvature of $\Om$ with respect to the outer unit normal $\nu_\Om$ to $\Om$, with the convention that $H_{B_r(x)}=n/r$ if $B_r(x)$ is the ball of radius $r$ in $\R^{n+1}$ centered at a point $x$. In \cite{ciraolomaggi2017} the {\it Alexandrov's deficit} $\de^*(\Om)$ of $\Om$
\begin{equation}
  \label{delta omega}
\de^*(\Om)=\Big\|\frac{H_\Om}{H_\Om^0}-1\Big\|_{C^0(\pa\Om)}\qquad H^0_\Om=\frac{n\,P(\Om)}{(n+1)|\Om|}\,.
\end{equation}
is introduced as a measure of how far $\Om$ is from being a critical point in the Euclidean isoperimetric problem. It is then proven that, if $\{\Om_h\}_{h\in\N}$ is a sequence of smooth bounded open sets in $\R^{n+1}$ normalized to have $H_\Om^0 =n$ and satisfying, for some $L\in\N$ and $\s\in(0,1)$,
\[
P(\Om_h)\le (L+\s)\, P(B_1)\qquad\forall h\in\N\,
\]
and if
\begin{equation}
  \label{small def intro}
  \lim_{h\to\infty}\de^*(\Om_h)=0\,,
\end{equation}
then there exists an open set $\Om$ consisting of the union of at most $L$ disjoint balls of radius one such that\begin{equation}
  \label{due}
  \lim_{h\to\infty}\big|P(\Om_h)-P(\Om)\big|+|\Om_h\Delta\Om|=0\,.
\end{equation}
Allard's monotonicity formula \cite[Section 17]{SimonLN} can then be exploited to deduce Hausdorff convergence (therefore, if the sets $\Om_h$ are connected, then the balls in $\Om$ are mutually tangent). Then, by exploiting Allard's regularity theorem \cite[Section 23]{SimonLN} and a calibration-type argument, one obtains the $C^{1,\a}$-convergence of $\pa\Om_h$ to $\pa\Om$ away from the tangency points of the limit balls. A quantitative analysis is also possible, both in the bubbling case (with non-sharp decay rates, see \cite{ciraolomaggi2017}) and in the one-bubble case $L=1$ (with sharp decay rates, see \cite{krummelmaggi}).

This type of compactness result is crucial for addressing the shape of critical points in capillarity problems. To illustrate this point, consider the capillarity energy
\begin{equation}
  \label{capillarity energy}
  P(\Om)+\int_\Om g(x)\,dx\,,\qquad |\Om|=v\,,
\end{equation}
of a liquid droplet occupying a region $\Om\subset\R^{n+1}$ of fixed volume $v$ under the action of a confinement potential $g$. If $v$ is suitably small with respect to $g$, then the surface energy $P(\Om)={\rm O}(v^{n/(n+1)})$ dominates over the potential energy term (of order ${\rm O}(v)$). By direct comparison with balls and by quantitative isoperimetry \cite{fuscomaggipratelli,FigalliMaggiPratelliINVENTIONES,CicaleseLeonardi} global minimizers are seen to be $L^1$-close to balls (quantitatively in terms of the size of $v$), and then a variational analysis proves they are actually $C^2$-close and thus convex \cite{FigalliMaggiARMA}. But direct comparison with balls is not available for addressing the shape of critical points, or even of local minimizers, and this is why a compactness theorem like the one proved in \cite{ciraolomaggi2017} is needed to get this analysis started. And indeed, in \cite[Corollary 1.4]{ciraolomaggi2017} it is shown that critical points of \eqref{capillarity energy} are quantitatively close close to compounds of mutually tangent balls with same radii, and that local minimizers are close to single balls, for $m$ small.

\subsection{The anisotropic setting and the elliptic compactness theorem}\label{section elliptic compactness}
In various situations of physical and geometric interest (see, e.g., the survey paper \cite{taylor78}) one is led to consider energies like \eqref{capillarity energy} with $P(\Om)$ replaced by an anisotropic surface energy of the form
\[
\F(\Om)=\int_{\pa\Om}F(\nu_\Om)\,d\H^n\,.
\]
Here $F:\S^n\to(0,\infty)$ is a {\it convex integrand}: namely, the one-homogenous extension of $F$ is convex on $\R^{n+1}$. As was proven in  \cite{taylor_roma_73,taylorstanford75,fonseca_wulff_rev,fonsecamuller_wulff,brothersmorgan}, the isoperimetric problem for $\F$ is uniquely solved by translations and scalings of the {\it Wulff shape $K_F$ of $F$}.
\begin{equation}
  \label{wulff shape}
  K_F=\bigcap_{\nu\in\S^n}\Big\{x\in\R^{n+1}:x\cdot\nu<F(\nu)\Big\}\,.
\end{equation}
This translates into the {\it Wulff inequality},
\begin{equation}
  \label{wulff inequality}
  \F(\Om)\ge (n+1)\,|K_F|^{1/(n+1)}\,|\Om|^{n/(n+1)}\qquad 0<|\Om|<\infty\,,
\end{equation}
where the right-hand side equals $\F(r\,K_F)$ for $r=(|\Om|/|K_F|)^{1/(n+1)}$ and where equality holds if and only if $\Om=x+r\,K_F$ for some $x\in\R^{n+1}$. The Wulff shape is always a bounded open convex set containing the origin; and, conversely, every bounded open convex set containing the origin is the Wulff shape of some $F$. Of particular interest is the case when $F$ is a {\it smooth elliptic integrand}, that is, $F\in C^\infty(\S^n;(0,\infty))$, and there exist constants $0<\l\le\Lambda<\infty$ such that, for every $\nu\in\S^n$,
\begin{equation}
  \label{elliptic integrand}
  \l\,\Id\le \nabla^2F(\nu)\le\Lambda\,\Id\qquad\mbox{on $\nu^\perp=T_\nu\S^n$}\,.
\end{equation}
In this setting one has a natural anisotropic extension of the notion of mean curvature. More precisely, the anisotropic mean curvature $H_\Om^F:\pa\Om\to\R$ of a set $\Om \subset \R^{n+1}$ with boundary of class $C^2$ is defined by
\begin{equation}
  \label{anisotropic mean curvature}
  H_\Om^F= \DIVV^{\pa\Om}\big(\na F\circ \nu_\Om\big)=\trace(\nabla^2F(\nu_\Om)\,\nabla\nu_\Om)\,,
\end{equation}
where $\Div^{\pa\Om}$ denotes the tangential divergence along $\pa\Om$, and $\nabla\nu_\Om$ is the second fundamental form of $\pa\Om$ with respect to $\nu_\Om$. For the Wulff shape, $K_F$ one has
\[
H^F_{K_F}=n\,.
\]
An anisotropic version of Alexandrov's theorem was shown in \cite{HeLiMaGe09}: if $\Om$ is a bounded smooth connected open set in $\R^{n+1}$ with constant anisotropic mean curvature, then $\Om=x+r\,K_F$ for some $x\in\R^{n+1}$ and $r>0$. In order to formulate a compactness theorem in this setting we introduce the scale invariant quantity
\begin{equation}
  \label{delta omega F star}
  \delta_F(\Om)=\Big(\frac1{\F(\Om)}\,\int_{\pa\Om}\Big|\frac{H_{\Om}^F}{H_{\Om}^{F,0}}-1\Big|^2\,F(\nu_{\Om})\,d\H^n\Big)^{1/2}
  \qquad H^{F,0}_\Om=\frac{n\,\F(\Om)}{(n+1)|\Om|}
\end{equation}
as an anisotropic generalization of \eqref{delta omega}, and then state the following theorem.

\begin{theorem}
  \label{thm main1} Let $F$ be a smooth elliptic integrand; see \eqref{elliptic integrand}, and let $\{\Om_h\}_{h\in\N}$ be a sequence of bounded open sets with smooth boundary normalized to have
  \[
  H_{\Om_h}^{F,0}=n\,.
  \]
  If, for some $L\in\N$, $\s\in(0,1)$, and $\k\in(0,1)$
  \[
   \sup_{h\in\N}\diam(\Om_h)<\infty\,,\qquad \sup_{h\in\N}\F(\Om_h)\le (L+\s)\,\F(K_{F})\,,
   \qquad\mbox{$H^F_{\Om_h}\ge \k\,n$ on $\pa\Om_h$}\,,
  \]
  and if
  \begin{equation}
    \label{small def main theorem}
     \lim_{h\to \infty} \de_{F}(\Om_h)=0\,,
  \end{equation}
  then there exists an open set $\Om$ consisting of the union of at most $L$-many disjoint translations of $K_F$, such that, up to translations and up to extracting subsequences,
  \begin{equation}
    \label{convergence main theorem}
     \lim_{h\to\infty}\big|\F(\Om_h)-\F(\Om)\big|+|\Om_h\Delta\Om|=0\,.
  \end{equation}
\end{theorem}

\begin{remark}
  {\rm Setting $\de(\Om)=\de_F(\Om)$ for $F\equiv 1$, and recalling the definition \eqref{delta omega} of $\de^*(\Om)$, we obviously have that $\de(\Om)\le\de^*(\Om)$, with $H_\Om\ge \k\,H^0_\Om$ provided $\de^*(\Om)\le1-\k$. In particular, Theorem \ref{thm main1} contains the fact that \eqref{small def intro} implies \eqref{due}, i.e., the key conclusion of \cite[Theorem 2.4]{ciraolomaggi2017}. As explained in Section \ref{section overview}, passing from the $C^0$ to the $L^2$ deficit is non-trivial because a key argument in the proof of \cite[Theorem 2.4]{ciraolomaggi2017} is a sliding argument based on the maximum principle for the mean curvature operator (see the argument right after \cite[Equation (2.48)]{ciraolomaggi2017}). For this kind of argument to work, it is crucial that whichever the contact point produced in the sliding argument is, the constant mean curvature deficit contains information at that point. This works naturally when using the $C^0$-deficit $\de^*(\Om)$, but it is clearly more delicate for the $L^2$-deficit $\de(\Om)$. We bypass this problem by exploiting the vanishing deficit assumption for obtaining a family of Pohozaev's-type identities of different homogeneities; see in particular step six in the proof of Theorem \ref{thm main}, Section \ref{section proof final}. We also notice that considering the $L^2$-deficit in this kind of problem is not merely done for the sake of generality, but is particularly significant in view of possible applications to the analysis of mean curvature flows.}
\end{remark}

We now describe some additional aspects of the proof of Theorem \ref{thm main1}. As we are going to see, a key tool will be a new potential theoretic proof of the anisotropic Heintze-Karcher inequality, done in the spirit of \cite{Ros}.

Let us first recall that the classical Heintze-Karcher inequality states that if $\Om$ is an open bounded connected set with smooth boundary and positive mean curvature, then
\begin{equation}
  \label{hk inq}
  \int_{\pa\Om}\frac{n}{H_\Om}\,d\H^n\ge(n+1)|\Om|\,,
\end{equation}
with equality if and only if $\Om$ is a ball. This inequality has been exploited by many authors \cite{Ros,montielros,brendlealex} as an effective starting point for proving the classical Alexandrov's theorem and its generalizations to higher order curvatures and to non-Euclidean ambient spaces. The relation is seen, also in the context compactness problems, if one considers the simple inequality
\begin{equation}
  \label{etadeltastar}
  \eta(\Om)\le\de^*(\Om)
\end{equation}
between the Alexandrov's deficit \eqref{delta omega} and the {\it Heintze-Karcher deficit}
\[
\eta(\Om)=1-\frac{(n+1)|\Om|}{\int_{\pa\Om}(n/H_\Om)\,d\H^n}\,.
\]
In particular, \eqref{etadeltastar} says that if $\Om$ has constant mean curvature, then $\Om$ is an equality case in \eqref{hk inq} (and thus a ball).

The compactness result obtained in \cite{ciraolomaggi2017} starts from Ros' proof \cite{Ros} of \eqref{hk inq}, which exploits the celebrated Reilly's identity \cite{reilly} in order to relate the Heintze-Karcher deficit and the torsion potential $u$ of $\Om$, i.e., the unique solution of
\[
\left\{\begin{split}
  \Delta u=1\qquad\mbox{in $\Om$}\,,
  \\
  u=0\qquad\mbox{on $\pa\Om$}\,.
\end{split}\right .
\]
For instance, a key estimate on $u$ in terms of $\eta(\Om)$ implied by Ros' argument takes the form
\begin{eqnarray}\label{ros bound}
  C(n)\,|\Om|\,\eta(\Om)\ge \int_{\Om}\Big|\nabla^2 u-\frac{\Id}{n+1}\Big|^2
\end{eqnarray}
where $|T|=(\sum_{ij}T_{ij}^2)^{1/2}$ for a matrix $T$. This inequality  expresses, in a rather direct way, the proximity of the torsion potential of $\Om$ to the torsion potential $(|x-x_0|^2-r^2)/2(n+1)$ of a suitable ball $B_r(x_0)$. The other known proofs of \eqref{hk inq}, namely \cite{montielros,brendlealex}, provide control of the almost-umbilicality of $\pa\Om$ in terms of $\eta(\Om)$. As explained in detail in \cite[Appendix]{ciraolomaggi2017}, almost-umbilicality is more difficult to exploit than direct information on the torsion potential to get compactness results in this setting. This is reflected in the fact that the current results concerning almost-umbilicality do not describe bubbling phenomena; see \cite{delellismuller1,delellismuller2,perez}.

The proof of Alexandrov's theorem for elliptic integrands in \cite{HeLiMaGe09} is based on an anisotropic version of \eqref{hk inq}, which states that if $F$ is smooth and elliptic and $\Om$ is an open bounded connected set with smooth boundary and positive $H^F_\Om$, then
\begin{equation}
  \label{hk inq anisotropic}
  \int_{\pa\Om}\frac{n}{H^F_\Om}\,F(\nu_\Om)\,d\H^n\ge(n+1)|\Om|\,,
\end{equation}
with equality if and only if $\Om=x+r\,K_F$ for some $x\in\R^{n+1}$ and $r>0$. Alternative proofs of \eqref{hk inq anisotropic} are obtained in \cite{MaXi13} and \cite{xia2016abp}. These arguments provide a control on the anisotropic almost-umbilicality of $\pa\Om$ in terms of an anisotropic Heintze-Karcher deficit. Anisotropic almost-umbilicality has been recently addressed in \cite{derosagiuff} under a convexity assumption, which of course prevents the possibility of bubbling into multiple Wulff shapes.

Our approach to Theorem \ref{thm main} will pass through a new proof of \eqref{hk inq anisotropic}, based on an adaptation of Ros' argument that allows us to control the suitable anisotropic variant of the torsion potential in terms of the {\it anisotropic Heintze-Karcher deficit $\eta_F$} of $\Om$,
\[
\eta_F(\Om)=1-\frac{(n+1)|\Om|}{\int_{\pa\Om}(n\,F(\nu_\Om)/H^F_\Om)}\,.
\]
(We consider  $\eta_F(\Om)$ only for sets with $H^F_\Om>0$ on $\pa\Om$, and we also notice that $\eta_F(\Om)\le\de_F(\Om)/\kappa$ provided $H_\Om^F\geq \kappa H_\Om^{F,0}$; see Lemma \ref{lemma magenta} below.) The right notion of anisotropic torsion potential is found by solving
\[
\left\{\begin{split}
  \Delta_Fu=1\qquad\mbox{in $\Om$}\,,
  \\
  u=0\qquad\mbox{on $\pa\Om$}\,.
\end{split}\right .
\]
where $\Delta_F$ denotes the Finslerian Laplace operator, defined by
\[
\Delta_Fu=\Div(\nabla(F^2/2)(\nabla u))\,.
\]
The operator is non-smooth at $\{\nabla u=0\}$, and particular attention must be paid in many parts of the argument in managing the critical set of $u$.

We now make some more technical comments on the proof, which should also be interesting in connection with our subsequent discussion of the role of ellipticity in the compactness theorem (see Section \ref{section weak alexandrov} below). As already mentioned, Ros' proof of \eqref{ros bound} crucially exploits Reilly's identity \cite{reilly}. Similarly, in Proposition \ref{proposition reilly} below, we prove the following anisotropic version of Reilly's identity (which, apparently, has not been previously stated in the literature):
\begin{equation}\label{aniso reilly introoo}
\int_\Omega (\Delta_F u)^2 -\trace\left((\nabla(\nabla_Fu))^2\right)= \int_{\pa \Om} H_\Om^F \, F(\nabla u)^2 F(\nu_\Om)\, d\H^n\,,
\end{equation}
where $\nabla_Fu=\nabla (F^2/2)(\nabla u)$. Armed with \eqref{aniso reilly introoo} and with a global Lipschitz estimate for $u$ {\it independent of the ellipticity constants of $F$} (see Proposition \ref{proposition lipschitz estimate}), we generalize \eqref{ros bound} to
\begin{eqnarray}\label{ros bound aniso}
  C(n)\,|\Om|\,\eta_F(\Om)\ge \int_{\Om}\Big\|\nabla(\nabla_Fu)-\frac{\Id}{n+1}\Big\|^2\,.
\end{eqnarray}
Here we set $\|A\|^2=\sum_{i=1}^{n+1} \mu_i^2$ where $\mu_i$ are the eigenvalues of and $A$. Unless we are in the isotropic case, the matrix $A=\nabla(\nabla_Fu)$ (which satisfies $A=\nabla^2F(\nabla u)\nabla^2 u$ on $\{\nabla u\ne0\}$) is not symmetric, but is real diagonalizable. In particular \eqref{ros bound aniso} does not provide direct control on the norm of $A$. However, such an estimate can be expressed if one allows the ellipticity constants of $F$ to appear in the estimate, which then takes the form
\begin{eqnarray}\label{ros bound aniso elliptic}
  C(n)\,|\Om|\,\frac{\Lambda}{\lambda}\,\eta_F(\Om)\ge \int_{\Om}\Big|\nabla(\nabla_Fu)-\frac{\Id}{n+1}\Big|^2\,.
\end{eqnarray}
This simple but delicate point is the only step of the proof of Theorem \ref{thm main1} where the ellipticity assumption is used.

\subsection{Critical points and local minimizers of anisotropic energies}\label{section local min} A natural and important application of Theorem \ref{thm main1} is the quantitative characterization of critical points and local minimizers of anisotropic surface energies under the action of a confining potential and with a small volume constraint. More precisely, the kind of energy we consider takes the form
\[
\E(\Om)=\F(\Om)+\int_{\Om} g\,,
\]
for a convex integrand $F$ and a smooth potential $g:\R^{n+1}\to\R$. When $g(x)\to\infty$ as $|x|\to\infty$ volume-constrained minimizers of $\E$ exist for every volume $v$. Of particular relevance in capillarity and phase transition models is the case when $v$ is small, and thus the surface energy $\F$ dominates the minimization.

We say that a bounded set of finite perimeter $\Om$ is a {\it volume-constrained critical point} of $\E$ if
\begin{equation}
\label{critical point of E 0}
\frac{d}{dt}\bigg|_{t=0}\E(f_t(\Om))=0
\end{equation}
whenever $f_t$ is a curve of diffeomorphisms such that $f_0=\Id$ and $|f_t(\Om)|=|\Om|$ for every $t$ in a neighborhood of $t=0$. We say that $\Om$ is a {\it volume-constrained $r_0$-local minimizer of $\E$} if there exists $r_0>0$ such that
\begin{equation}
  \label{locmin hp}
  \E(\Om)\le\E(\Om')\qquad
\mbox{whenever $|\Om'|=|\Om|$ and $\Om\Delta\Om'\cc I_{|\Om|^{1/(n+1)}\,r_0}(\pa\Om)$}\,.
\end{equation}
Here, $I_r(E)=\{x\in\R^{n+1}:\dist(x,E)<r\}$ denotes the $r$-neighborhood of a set $E\subset\R^{n+1}$. Obviously, local minimizers are critical points.

As an application of Theorem \ref{thm main1}, we analyze volume-constrained local minimizers and critical points of $\E$.

\begin{theorem}
  \label{thm locmin} Let $F$ be a smooth elliptic integrand, $g:\R^{n+1}\to\R$ be a smooth function, $M>0$, and let $\Om\subset B_M$ be an open connected set with smooth boundary such that
  \[
  \frac{\F(\Om)^{n+1}}{|\Om|^n}\le M\,,\qquad\frac{\diam(\Om)}{|\Om|^{1/(n+1)}}\le M\,.
  \]
  Then the following holds:

  \medskip

  \noindent (i) for every $\e>0$ there exists $v_\e=v_\e(\e,n,M,F,g)>0$ such that if $\Om$ is a volume-constrained critical point of $\E$ with
  \[
  |\Om|<v_\e\,,
  \]
  then
  \[
  \Big|\Om^*\Delta \bigcup_{i=1}^L(x_i+K_F)\Big|\le\e\qquad\mbox{where}\quad \Om^*=\frac{H^{F,0}_\Om}{n}\,\Om\,,
  \]
  for some $L\ge 1$ (which is a priori bounded from above in terms of $F$ and $M$) and $\{x_i\}_{i=1}^L\subset\R^{n+1}$ such that the sets $\{x_i+K_F\}_{i=1}^L$ are mutually disjoint.

    \medskip

  \noindent (ii) given $r_0>0$, there exists $v_0=v_0(r_0,n,M,F,g)>0$ such that, if $\Om$ is a volume-constrained $r_0$-local minimizer of $\E$ with
  \[
  |\Om|<v_0\,,
  \]
  then there exists $x_0\in\R^{n+1}$ such that, setting $K=x+s\,K_F$ with $s=(|\Om|/|K_F|)^{1/(n+1)}$, one has
  \begin{equation}
    \label{locmin L1 close}
      \frac{|\Om\Delta K|}{|\Om|}\le C(n,F,g,M)\,|\Om|^{1/(n+1)}\,,
  \end{equation}
  \begin{equation}
    \label{locmin hd close}
      \frac{\hd(\pa\Om,\pa K)}{|\Om|^{1/(n+1)}}\le C(n,F,g,M)\,|\Om|^{1/(n+1)^2}\,.
  \end{equation}
\end{theorem}

This result was shown for volume-constrained global minimizers of $\E$
in \cite[Theorem 1]{FigalliMaggiARMA}.
The key difference between the case of global minimizers
and the case of critical points/local minimizers addressed here is that \eqref{critical point of E 0} and \eqref{locmin hp} do not immediately allow the comparison of the energy of $\Om$ with that of a Wulff shape with same volume. This direct comparison was the key step in \cite{FigalliMaggiARMA} to exploit the quantitative Wulff inequality from \cite{FigalliMaggiPratelliINVENTIONES} and to deduce the $L^1$-proximity of any global minimizer to a Wulff shape.

Volume-constrained critical points of $\E$ turn out be volume-constrained almost-critical points of $\F$ thanks to a variational argument, so Theorem \ref{thm main1} allows one to deduce that $\Om$ is close in volume to a finite family of Wulff shapes. In the case of local minimizers, using volume density estimates we can show this family consists of a single Wulff shape, and then $\pa\Om$ is close to the boundary of this Wulff shape in Hausdorff distance. But this implies that a Wulff shape of the appropriate volume is an admissible competitor in \eqref{locmin hp}, which in turn enables us to exploit the quantitative Wulff inequality and obtain \eqref{locmin L1 close} and \eqref{locmin hd close}; see Section \ref{section locmin}.

In \cite[Theorem 2]{FigalliMaggiARMA} it was proven that a volume-constrained global minimizer $\Om$ of $\E$ is actually convex with
\begin{equation}
\label{valid for}
  \Big\|\nabla^2F(\nu_{\Om^*})\nabla\nu_{\Om^*}-\Id\Big\|_{C^0(\pa\Om^*)}\le C(n,F,g)\,|\Om|^{2/(n+3)}\,.
\end{equation}
(Let us recall here that $\nabla^2F(\nu_{K_F})\nabla\nu_{K_F}=\Id$ on $\pa K_F$.)
In our setting of connected volume-constrained local minimizers, once \eqref{locmin hd close} is proven, one can also repeat all the remaining analysis from \cite{FigalliMaggiARMA}. In particular, arguing as in \cite[Appendix C]{FigalliMaggiARMA}, one can show that $\pa\Om$ is a $C^1$-small normal deformation of $\pa K$ with quantitative bounds on the $C^1$-norm of the normal deformation in terms of explicit powers of $|\Om|$. Then, one can also repeat the argument from \cite[Theorem 13]{FigalliMaggiARMA} to show that \eqref{valid for} holds and thus that $\Om$ is convex. As this part of the argument would be identical to that of \cite{FigalliMaggiARMA}, we omit it and simply remark that, thanks to Theorem \ref{thm main1}, all the small-volume regime properties of volume-constrained global minimizers of $\E$ proved in \cite{FigalliMaggiARMA} hold for connected local minimizers as well.

\subsection{Weak Alexandrov's theorem for convex integrands}\label{section weak alexandrov}  A way to look at the classical Alexandrov's theorem is to consider it as a {\it convexity property} of the volume-constrained perimeter functional; Alexandrov's theorem says that the only critical points of the volume-constrained perimeter are its global minimizers.  The property that global minimizers are the only critical points is a characteristic consequence of convexity.

The same can be said about the anisotropic Alexandrov's theorem for smooth elliptic integrands \cite{HeLiMaGe09}. Once again, we have an energy functional whose only critical points are its global minimizers. A natural question is whether this property depends on the assumption that the anisotropy is smooth and elliptic. Indeed, anisotropic energies that fail to be either smooth or elliptic (or both) are of great interest in applications, and the anisotropic isoperimetric problem is totally unaffected by the lack of smoothness or of ellipticity (recall that Wulff shapes are the unique isoperimetric sets of every convex integrand). Therefore one conjectures that Alexandrov's theorem should hold, in some proper formulation, for every anisotropic energy. This conjecture seems open in ambient space dimension $n+1\ge 3$. A result of Morgan \cite{MorganPlanar}, which proves that Wulff shapes are the only anisotropic critical immersion of a curve in the plane, is an indication in favor of this conjecture. Of course, there is a substantial difference between the planar case and higher dimensions, because in the planar case criticality implies convexity of the connected components.

With this premise in mind, we now discuss a generalization of Theorem \ref{thm main1}, namely Theorem \ref{thm main} below, which serves as a sort of {\it weak version of Alexandrov's theorem valid for arbitrary convex integrands}.

\begin{theorem}[Weak anisotropic Alexandrov theorem]
  \label{thm main} Let $\{F_h\}_{h\in\N}$ be a sequence of smooth elliptic integrands that converges pointwise to a limit function F, and assume that, for some positive constants $m$, $M$, $\l_h$ and $\Lambda_h$ ($m$ and $M$ independent of $h$)
  \begin{equation}\label{smooth elliptic approx}
  \begin{split}
    m\le F_h\le M&\qquad\mbox{on $\S^n$}\,,
    \\
    \l_h\,\Id\le \nabla^2 F_h(\nu)\le \Lambda_h\,\Id&\qquad\mbox{on $\nu^\perp$, $\forall\nu\in\S^n$}\,.
  \end{split}
  \end{equation}
  Then the following holds: If $\{\Om_h\}_{h\in\N}$ is a sequence of bounded open sets with smooth boundary, normalized so that $H_{\Om_h}^{F_h,0}=n$ with
  \begin{equation}
    \label{xyzttt}
      \mbox{$H^{F_h}_{\Om_h}\ge \kappa \,n$ on $\pa\Om_h$}\,,\qquad
      \sup_{h\in\N}\diam(\Om_h)<\infty\,,\qquad \sup_{h\in\N}\frac{\F_h(\Om_h)}{\F_h(K_{F_h})}\le L+\s\,
  \end{equation}
  for some $\kappa>0$,
  $L\in\N$ and $\s\in(0,1)$, and if
  \begin{equation}
    \label{small def main theorem}
      \lim_{h\to\infty}\max\Big\{\frac1{\l_h^2}\,,\frac{\Lambda_h}{\lambda_h}\Big\}\,\eta_{F_h}(\Om_h)+\de_{F_h}(\Om_h)=0\,,
  \end{equation}
  then there exists an open set $\Om$, which is the disjoint union of at most $L$-many translations of $K_F$ such that, up to translations and up to extracting subsequences,
  \begin{equation}
    \label{convergence main theorem}
     \lim_{h\to\infty}\big|\F_h(\Om_h)-\F(\Om)\big|+|\Om_h\Delta\Om|=0\,.
  \end{equation}
\end{theorem}

\begin{remark}
  {\rm The assumption \eqref{small def main theorem} is implied by $\max\{1/\l_h^2,\Lambda_h/\l_h\}\,\de_{F_h}(\Om_h)\to 0$
  because $\eta_F\le\de_F/\kappa$ (see Lemma~\ref{lemma magenta} below).
  In particular, Theorem \ref{thm main1} is a corollary of Theorem \ref{thm main} if one takes $F$ to be smooth and elliptic and $F_h\equiv F$.}
\end{remark}

The interpretation of Theorem \ref{thm main} as a weak version of Alexandrov's theorem for general convex integrands is the following. Clearly, every convex integrand $F$ can be approximated by smooth elliptic integrands $F_h$ as in \eqref{smooth elliptic approx}, as the ratio $\max\{1/\l_h^2,\l_h/\Lambda_h\}$ is allowed to vanish in the limit as $h\to\infty$.  Theorem \ref{thm main} thus asserts that (unions of) Wulff shapes of $F$ are the only possible accumulation points of sequences of almost-critical points of the approximating $\F_h$ with anisotropic Heintze-Karcher deficit vanishing faster than the rate of degeneracy of the ellipticity constants. This fast convergence assumption may be purely technical in nature. In our argument, it arises exactly in the derivation of \eqref{ros bound aniso elliptic} from \eqref{ros bound aniso}, as previously explained.

Theorem \ref{thm main} immediately leads to the following delicate question: given a convex integrand $F$ and a bounded open set $\Om$ with Lipschitz boundary that is a critical point for the volume-constrained anisotropic energy $\F$, is it possible to construct a sequence of smooth elliptic integrands $F_h$ and smooth open sets $\Om_h$ such that $F_h$ converges to $F$ as in \eqref{smooth elliptic approx} and $\Om_h$ satisfies \eqref{xyzttt}, \eqref{small def main theorem} and $|\Om_h\Delta\Om|\to 0$? Whenever this is possible, of course, Theorem \ref{thm main} implies that finite unions of Wulff shapes are the only critical points for the convex integrand $F$.

\bigskip

\noindent{\bf Acknowledgments.} RN supported by the NSF Graduate Research Fellowship under Grant DGE-1110007. FM, RN, and CM supported by the NSF Grants DMS-1565354 and DMS-1361122.

\section{Basic facts on integrands}\label{section basic}  We recall without proof some standard facts about convex integrands and their gauge functions. By {\it convex integrand} we refer to a positive function $F$ on $\S^n$ with a convex one-homogeneous extension to $\R^{n+1}$. The {\it gauge function of $F$} is defined by
\begin{equation}
  \label{def gauge of F}
  F_*(x) = \sup \big\{x \cdot \nu :F(\nu)<1\big\} \qquad x \in \R^{n+1}\,
\end{equation}
and is itself a convex integrand. By convexity of $F$, one has $(F_*)_*=F$. Moreover, the {\it Fenchel inequality} holds:
\begin{equation}
  \label{fenchel inequality}
  x\cdot\nu\le F_*(x)\,F(\nu)\qquad\forall x,\nu\in\R^{n+1}\,.
\end{equation}
If we denote the minimum and the maximum values of $F$ on $\S^n$ by
\begin{equation}
  \label{mm1}
  m_F=\min_{\S^n}\,F,\qquad M_F=\max_{\S^n}F\,,
\end{equation}
then we have
\begin{equation}
  \label{mm2}
  m_F\le|\nabla F(\nu)|\le M_F
\end{equation}
whenever $F$ is differentiable at $\nu$ (that is to say, at almost every $\nu\in\R^{n+1}$), and moreover
\begin{equation}
  \label{mFstar MFstar}
  m_{F_*}=\frac1{M_F}\qquad M_{F_*}=\frac1{m_F}\,.
\end{equation}
By \eqref{wulff shape}, \eqref{def gauge of F} and since $(F_*)_*=F$,
\begin{equation}
  \label{wulff shapes as level sets}
  K_F=\{F_*<1\}\qquad K_{F_*}=\{F<1\}\,.
\end{equation}
Recall that the subdifferential of $F$ at $\nu\in\R^{n+1}$ is defined by
\[
\pa F(\nu)=\Big\{x\in\R^{n+1}:F(w)\ge F(\nu)+x\cdot(w-\nu)\quad\forall w\in\R^{n+1}\Big\}\,,
\]
so $\pa F(\nu)=\{\nabla F(\nu)\}$ if $F$ is differentiable at $\nu$. We immediately see from \eqref{wulff shape} that
\[
K_F=\pa F(0)\,,\qquad K_{F_*}=\pa F_*(0)
\]
and hence, by one-homogeneity of $F$,
\begin{equation}
  \label{boundary wulff as paF}
  \pa K_F=\{F_*=1\}=\bigcup_{\nu\ne 0}\pa F(\nu)\,,\qquad \pa K_{F_*}=\{F=1\}=\bigcup_{x\ne 0}\pa F_*(x)\,.
\end{equation}
(Notice that we use the same symbol $\pa$ both for topological boundaries and for subdifferentials.)

If $F\in C^1(\S^n)$ and is strictly convex, then $F_*\in C^1(\S^n)$ and is strictly convex, and therefore $\pa F(\nu)=\{\nabla F(\nu)\}$ and $\pa F_*(x)=\{\nabla F_*(x)\}$ for every $\nu,x\ne 0$. Under these assumptions, using \eqref{boundary wulff as paF}, one deduces the useful properties
\begin{equation}
  \label{F and its gauge}
  \begin{split}
  F(\nabla F_*(x))=1\,,& \qquad F_*(x)\,\nabla F(\nabla F_*(x))=x\,\qquad \forall x\ne 0\,,\\
  F_*(\nabla F(\nu))=1\,,& \qquad F(\nu)\nabla F_*(\nabla F(\nu))= \nu \qquad\forall \nu\ne 0\,.
  \end{split}
\end{equation}
A short explanation of the vector identities in \eqref{F and its gauge} is as follows. By $\{F_*<1\}$, the normal to $K_F$ at $\nabla F(\nu)\in\pa K_F$ (for $\nu\ne 0$) is parallel to $\nabla F_*(\nabla F(\nu))$. At the same time, by exploiting \eqref{wulff shape} and the one-homogeneity of $F$, we see that the normal to $K_F$ at $\nabla F(\nu)$ is parallel to $\nu$ itself. Thus $\nabla F_*(\nabla F(\nu))=\a\,\nu$ and then one finds $\a=1/F(\nu)$ thanks to $F(\nabla F_*(x))=1$ for $x\ne 0$.

We conclude with some remarks about the functions $F^2/2$ and $F_*^2/2$. First, an important consequence of \eqref{F and its gauge} is that
\begin{equation}
  \label{inverse}
  \nabla(F_*^2/2)=[\nabla(F^2/2)]^{-1} \qquad \text{on } \R^{n+1}\setminus \{0\}\,.
\end{equation}
Indeed, for every $z\in\R^{n+1}$,
  \begin{eqnarray*}
    \nabla(F_*^2/2)(\nabla(F^2/2)(z))&=&F_*\big(F(z)\nabla F(z)\big)\,\nabla F_*\big(F(z)\nabla F(z)\big)
    \\
    &=&F(z)\,F_*(\nabla F(z))\,\nabla F_*(\nabla F(z))=F(z)\,\frac{z}{F(z)}=z\,.
  \end{eqnarray*}
Moreover, \eqref{inverse} holds for generic convex integrands in the sense that
\begin{equation}\label{eq:invertingsubdifferentials}
z\in \partial (F^2/2)(x) \iff x\in \partial (F_*^2/2)(z)\,.
\end{equation}
When $F\in C^\infty(\S^n)$ is an elliptic integrand (so that \eqref{elliptic integrand} holds with constants $\l$ and $\Lambda$),
one has that $F^2/2\in C^{1,1}(\R^{n+1})\cap C^\infty(\R^{n+1}\setminus\{0\})$ with
\begin{eqnarray}\label{GF formulas}
\nabla (F^2/2)(\nu)&=&\left\{
\begin{split}
  &F(\nu)\,\nabla F(\nu)\,,\qquad\mbox{if $\nu\ne 0$}\,,
  \\
  &0\,,\hspace{1.8cm}\qquad\mbox{if $\nu=0$}\,,
\end{split}
\right .
\\\nonumber
\nabla^2(F^2/2)(\nu)&=&\nabla F(\nu)\otimes\nabla F(\nu)+F(\nu)\nabla^2F(\nu)\qquad\mbox{if $\nu\ne 0$}\,,
\end{eqnarray}
and
\begin{equation}
  \label{elliptic integrand GF}
  \l_*\,\Id\le \nabla^2 (F^2/2)(\nu)\le \Lambda_*\,\Id\qquad\mbox{on $\R^{n+1}\setminus\{0\}$}
\end{equation}
for positive constants $\l_*$ and $\Lambda_*$. By \eqref{GF formulas}, we see that we can take
\[
\l_*=m_F\,\min\{\l,m_F\}\qquad \Lambda_*=2 M_F\max\{M_F,\Lambda\}\,.
\]
We have the identity
\begin{equation}\label{eqn: hessian inverses}
\na^2(F^2/2)(\nu) \circ \na^2(F_*^2/2)(\na F(\nu)) = \Id\qquad\forall \nu\in\S^n\,.
\end{equation}
To conclude this section, we note that will sometimes use the shorthand
\begin{equation}\label{eqn: cahn}
\na_Fu(x) = \na(F^2/2)(\na u(x))
\end{equation}

\section{Proof of the main theorem} This section contains the proof of Theorem \ref{thm main}.  Section \ref{section distance function} and Section \ref{section torsion potential} serve to introduce, the anisotropic signed distance function and the anisotropic torsion potential on a smooth bounded open set respectively. In Section \ref{section anisotropic reilly} we prove an anisotropic version of Reilly's identity, while in Section \ref{section lipschitz estimate} we obtain a Lipschitz estimate on the anisotropic torsion potential that is interestingly independent of ellipticity. In Section \ref{section hk with bounds} we exploit the torsion potential to give a proof of the anisotropic Heintze-Karcher inequality of \cite{HeLiMaGe09} in the spirit of Ros' argument. The key result is identity \eqref{hk anisotropic plus}, which relates the gap in the  anisotropic Heintze-Karcher inequality to the properties of the torsion potential. Together with the Lipschitz estimate, this is a key fact to obtain compactness. This is finally done in Section \ref{section proof final}, where we prove Theorem \ref{thm main}.

\subsection{The anisotropic signed distance function}\label{section distance function} In the next proposition we collect some useful facts on the anisotropic signed distance function from a bounded open smooth set.

\begin{proposition}\label{proposition signed distance}
  Let $F$ be an elliptic integrand, let $\Om$ be a bounded open smooth set, and define the $F$-anisotropic signed distance function of $\Om$ by
  \begin{equation}
    \label{anisotropic signed distance}
      \g(x)=\inf\big\{F_*(x-y):y\in\Om\big\}-\inf\big\{F_*(y-x):y\in\Om^c\big\}\,,\qquad x\in\R^{n+1}\,.
  \end{equation}
  Then there exists an open neighborhood $N$ of $\pa\Om$ such that $\g$ is smooth in $N$ and
  \begin{equation}
    \label{nablag}
      \nabla\g(y)=\frac{\nu_\Om(y)}{F(\nu_\Om(y))}\qquad\forall y\in\pa\Om\,.
  \end{equation}
\end{proposition}
\begin{proof} Note that $\g<0$ in $\Om$, $\g>0$ in $\Om^c$.
   Since $\Om$ has smooth boundary, it satisfies uniform exterior and interior balls conditions. Since $F$ is a smooth elliptic integrand, the Wulff shape $K_F=\{F_*<1\}$ is uniformly convex, and $K_F^-=\{x\in\R^{n+1}:F_*(-x)<1\}$ is uniformly convex as well. Combining these facts, we see that there exists $r>0$ with the following property: for every $y\in\pa\Om$, there exist $x_y\in\Om$ and $z_y\in\Om^c$ such that
  \begin{eqnarray}
    \label{xy}
    \big\{w\in\R^{n+1}:F_*(w-x_y)<r\big\}\subset\Om\,,&\qquad F_*(y-x_y)=r\,,
    \\
    \label{zy}
    \big\{w\in\R^{n+1}:F_*(-(w-z_y))<r\big\}\subset\Om^c\,,&\qquad F_*(-(y-z_y))=r\,;
  \end{eqnarray}
  see Figure \ref{fig distance}.
  \begin{figure}
    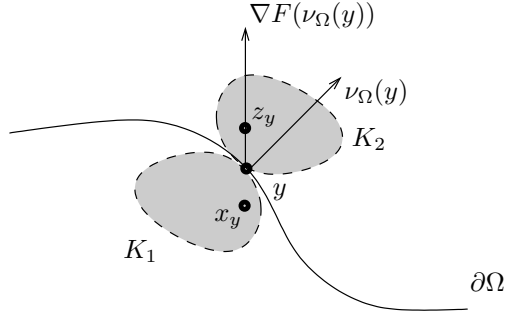\caption{{\small The uniformly convex sets $K_1=\{F_*(\cdot-x_y)<r\}$ and $K_2=\{F_*(z_y-\cdot)<r\}$ touch $\Om$ at $y$, respectively, from inside and from outside. The point $y$ lies on the segment joining centers $x_y$ and $z_y$, which is parallel to $\nabla F(\nu_\Om(y))$.}}\label{fig distance}
  \end{figure}
  We claim that $y-x_y$ and $z_y-y$ are parallel to $\nabla F(\nu_\Om(y))$. Indeed, by \eqref{xy},
  $\nabla F_*(y-x_y)=\a\,\nu_\Om(y)$ with $\a>0$; applying $F$ to both sides and exploiting \eqref{F and its gauge}, we find $\a\,F(\nu_\Om(y))=1$ and so
  \begin{equation}
    \label{xy1}
      \nabla F_*(y-x_y)=\frac{\nu_\Om(y)}{F(\nu_\Om(y))}\,.
  \end{equation}
  An analogous argument shows that $-\nabla F_*(z_y-y)=-\b\,\nu_\Om(y)$ for some $\b>0$, which satisfies $\b=1/F(\nu_\Om(y))$ thanks again to \eqref{F and its gauge}. Thus
  \begin{equation*}
    \label{zy1}
      \nabla F_*(z_y-y)=\frac{\nu_\Om(y)}{F(\nu_\Om(y))}\,.
  \end{equation*}
  Applying $\nabla F$ to both sides of \eqref{xy1} and \eqref{zy1} and taking \eqref{F and its gauge} into account, we find
  \begin{equation}
    \label{xyz}
      \frac{y-x_y}{F_*(y-x_y)}=\nabla F(\nu_\Om(y))=\frac{z_y-y}{F_*(z_y-y)}\,,
  \end{equation}
  as claimed. Now consider the projection map $p:\R^{n+1}\to\pa\Om$ defined by
  \begin{equation}\label{eqn: projection}
  p(x)=\left\{
  \begin{split}
    \big\{y\in\pa\Om:F_*(y-x)=-\g(x)\big\}\qquad\mbox{if $x\in\Om$}\,,
    \\
    \big\{y\in\pa\Om:F_*(x-y)=\g(x)\big\}\qquad\mbox{if $x\in\Om^c$}\,.
  \end{split}
  \right .
  \end{equation}
  Fix $y\in\pa\Om$. By the first claim, if $x\in[x_y,y]$ or $x\in[y,z_y]$, then $p(x)=\{y\}$.So there exists a constant $c>0$, depending only on $m_F$, $M_F$ and the $C^2$ norm of $F$, such that
  \[
  c\le\min\big\{|x_y-y|,|z_y-y|\big\}\,.
  \]
 Hence there exists $r_0$ such that if we set
  \[
  N=\bigcup_{y\in\pa\Om}\Big\{y+t\,\nabla F(\nu_\Om(y)):|t|<r_0\Big\}
  \]
  then for every $x\in N$ there exists a unique $y\in\pa\Om$ such that $p(x)=\{y\}$, with
  \[
  x=y+\g(x)\nabla F(\nu_\Om(y))\,.
  \]
  Exploiting the implicit function theorem as in \cite[Lemma 14.16]{GT}, we find that $\g$ is smooth on $N$. Now, if $x_0\in\Om$, then
  \[
  \g(x)\ge-F_*(p(x_0)-x)\qquad\forall x\in\Om\,,\qquad \qquad \g(x_0)=-F_*(p(x_0)-x_0)\,,
  \]
  so that $\nabla\g(x_0)=\nabla F_*(p(x_0)-x_0)$ for every $x_0\in\Om$. By \eqref{xy1},
  \[
  \nabla\g(x_0)=\frac{\nu_\Om(p(x_0))}{F(\nu_\Om(p(x_0)))}\,.
  \]
  Letting $x_0\to y\in\pa\Om$, we obtain \eqref{nablag}.
\end{proof}

\subsection{The anisotropic torsion potential}\label{section torsion potential} Let $F$ be an elliptic functional and $\Om$ a bounded open set with smooth boundary. The {\it $F$-anisotropic torsion potential of $\Om$} is the unique minimizer $u\in H^1_0(\Om)$ of the strictly convex functional
\[
\int_\Om \frac{F^2(\nabla u)}{2}+u\,.
\]
The Euler-Lagrange equation
\begin{equation}
  \label{anisotropic torsion potential weak}
  \int_\Om \nabla (F^2/2)(\nabla u)\cdot\nabla\vphi=-\int_\Om\vphi\qquad\forall\vphi\in C^1_c(\Om)\,
\end{equation}
holds, that is to say, $u$ is a distributional solution of
\[
\left\{
\begin{split}
  \Delta_Fu=1\,,&\qquad\mbox{on $\Om$}
  \\
  u=0\,,&\qquad\mbox{on $\pa\Om$}\,.
\end{split}
\right .
\]
Here, we have introduced the {\it Finslerian Laplace operator}
\[
\Delta_Fu=\Div(\nabla (F^2/2)(\nabla u))\,.
\]
Notice that the Finslerian Laplace operator satisfies the classical {\it comparison principle}:
\begin{equation}
  \label{comparison principle}
  \left\{
  \begin{split}
    v_1,v_2\in H^1_0(\Om)&
    \\
    -\Delta_Fv_1\ge-\Delta_Fv_2&\qquad\mbox{on $\Om$}
    \\
    v_1\ge v_2&\qquad\mbox{on $\pa\Om$}
  \end{split}
  \right .
  \quad
  \Rightarrow
  \quad
  \mbox{$v_1\ge v_2$ on $\Om$.}
\end{equation}
Indeed, setting $w=\max\{v_2-v_1,0\}$ we find
  \begin{eqnarray*}
  0&\ge&\int_\Om w\,(-\Delta_Fv_2+\Delta_Fv_1)=\int_\Om\nabla w\cdot\big(\nabla (F^2/2)(\nabla v_2)-\nabla (F^2/2)(\nabla v_1)\big)
    \\
    &\ge&\l_*\int_{\{v_2>v_1\}}|\nabla v_2-\nabla v_1|^2
  \end{eqnarray*}
  thanks to \eqref{elliptic integrand GF}. Based on classical regularity arguments, one can prove that $u\in C^{1,\a}(\overline{\Om})\cap H^2(\Om)$ for some $\a\in(0,1)$, see, e.g. \cite[Proposition 2.3]{CianchiSalani}. The critical set
\[
\C=\big\{x\in\ov{\Om}:\nabla u(x)=0\big\}
\]
is thus closed in $\ov{\Om}$ and, since $F^2/2\in C^\infty(\R^{n+1}\setminus\{0\})$, $u$ is smooth in a neighborhood of each $x\in\ov{\Om}\setminus\C$. The following proposition collects some further properties of $u$.

\begin{proposition}\label{proposition some properties of u}
  The critical set $\C$ of the anisotropic torsion potential has Lebesgue measure zero. The torsion potential $u$ is negative in $\Om$ with positive outer normal derivative along $\pa\Om$. In particular, $u$ is smooth in a neighborhood of $\pa\Om$.
\end{proposition}

\begin{proof}
  {\it Step one}: The function $f:\R^{n+1}\to\R^{n+1}$ defined by $f(\nu)=\nabla (F^2/2)(\nu)$  is Lipschitz with $f(0)=0$. Setting $v=\nabla u\in H^1(\Om;\R^{n+1})$, we have $f\circ v\in H^1(\Om;\R^{n+1})$. In particular, \eqref{anisotropic torsion potential weak} implies that
  \begin{equation}
    \label{el strong 1}
      \trace\big(\nabla (f\circ v)\big)=1\qquad\mbox{a.e. on $\Om$}\,.
  \end{equation}
  Let us denote by $L_x$ the affine $n$-dimensional space $L_x=v(x)+\nabla v(x)[\R^{n+1}]$. By the Sobolev chain rule for vector valued Lipschitz functions \cite{ambrosiodalmasochainrule} (see also \cite[Theorem 1.1]{leonimorinichainrule}),  for almost every $x\in\Om$ we have that $(f\circ v)|_{L_x}$ is differentiable at $v(x)$, with
  \[
  \nabla (f\circ v)(x)=\nabla \big((f\circ v)|_{L_x}\big)(v(x))\circ\nabla v(x)\,.
  \]
  Since $\nabla v=0$ almost everywhere on $\C$, we conclude that $\nabla (f\circ v)=0$ almost everywhere on $\C$. This fact, combined with \eqref{el strong 1}, implies that $|\C|=0$.

  \bigskip

  \noindent {\it Step two}: We show that
  \[
  h(x) =\frac{F_*(x)^2}{2(n+1)}\qquad x\in\R^{n+1}\,,
  \]
  satisfies $\Delta_Fh=1$ on $\R^{n+1}$. Indeed $\nabla h(x)=F_*(x)\nabla F_*(x)/(n+1)$, so, by one-homogeneity of $F$ and by \eqref{F and its gauge}, we have
  \[
  F(\nabla h(x))=\frac{F\big(F_*(x)\nabla F_*(x)\big)}{n+1}=\frac{F_*(x)}{n+1}\,F(\nabla F_*(x))=\frac{F_*(x)}{n+1}\,.
  \]
  At the same time, by zero-homogeneity of $\nabla F$ and again by \eqref{F and its gauge},
  \[
  \nabla F(\nabla h(x))=\nabla F(\nabla F_*(x))=\frac{x}{F_*(x)}\qquad\forall x\ne 0\,.
  \]
  Thus $\nabla (F^2/2)(\nabla h(x))=x/(n+1)$ for $x\in \R^{n+1}\setminus\{0\}$, and the desired result follows.

  \bigskip

  \noindent {\it Step three}: We use translations of $h$ as barriers to prove Hopf's lemma and the negativity of $u$ in $\Om$. As seen in Proposition \ref{proposition signed distance}, for every $y_0\in\pa\Om$ we can find $x_0\in\Om$ such that
  \[
  \inf_{y\in\Om^c}F_*(y-x_0)=F_*(y_0-x_0)\,,\qquad \nabla F_*(y_0-x_0)=\frac{\nu_\Om(y_0)}{F(\nu_\Om(y_0))}\,.
  \]
  Corresponding to this choice of $x_0$, we have that
  \[
  \beta=\inf\Big\{h(y-x_0):y\in\Om^c\Big\}=h(y_0-x_0)\,,\qquad \nabla h(y_0-x_0)\cdot\nu_\Om(y_0)>0\,.
  \]
  If we set $v(x)=h(x-x_0)-\beta$ for $x\in\Om$, then $v\ge 0$ on $\pa\Om$, and $\Delta_Fv=1$ in $\Om$. By the comparison principle \eqref{comparison principle}, we find $v\ge u$ in $\Om$, and since $v(y_0)=u(y_0)=0$ with $\nabla v(y_0)\cdot\nu_\Om(y_0)>0$ we conclude that
  \[
  -\nu_\Om(y_0)\cdot\nabla u(y_0)=\lim_{t\to 0^+}\frac{u(y_0-t\nu_\Om(y_0))-u(y_0)}t\le
  \lim_{t\to 0^+}\frac{v(y_0-t\nu_\Om(y_0))-v(y_0)}t<0\,.
  \]
  This proves that $\nu_\Om\cdot\nabla u>0$ on $\pa\Om$. Since $\pa\Om=\{u=0\}$, we have that $|\nabla u|>0$ on $\pa\Om$. Therefore, because $u\in C^{1,\a}(\ov{\Om})$, we find $\dist(\C,\pa\Om)>0$. Since $u$ is smooth on $\ov{\Om}\setminus\C$, we conclude that $u$ is smooth in a neighborhood of $\pa\Om$. Thus, $u$ is strictly negative in a neighborhood of $\pa\Om$, and therefore on the rest of $\Om$ by the comparison principle.
%
%
  \end{proof}

\subsection{The anisotropic Reilly's identity}\label{section anisotropic reilly} The goal of this section is to prove an anisotropic variant of Reilly's identity. Let us recall that if $\Om$ is an open set with smooth boundary and $F$ is an elliptic integrand, then the $F$-anisotropic mean curvature of $\Om$ (with respect to the outer unit normal $\nu_\Om$) is defined as
\[
H^F_\Om=\Div^{\pa\Om}(\nabla F(\nu_\Om))\qquad\mbox{on $\pa\Om$}\,.
\]
The anisotropic mean curvature shares many basic properties of its isotropic counterpart. Of particular importance to us will be the anisotropic variant of the classical identity
\[
\Delta u-\nabla^2u[\nu_\Om,\nu_\Om]=|\nabla u|\,H_\Om\qquad\mbox{on $\pa\Om$}\,,
\]
which holds when $\pa\Om$ is the level set of a smooth function $u$ with non-vanishing gradient on $\pa\Om$. The anisotropic counterpart of this formula involves the Finslerian Laplace operator and takes the form
\begin{equation}\label{eqn: ARMA id}
\Delta_F u-\nabla^2u [ \na F(\na u) , \na F(\na u)]= F(\na u)\,H_\Om^F\qquad\mbox{on $\pa\Om$}\,.
\end{equation}
This formula is derived, under different conventions, in \cite[Theorem 3]{WangXiaARMA}. For the sake of clarity we recall the short proof. First, we claim that
\begin{equation}
  \label{wx1}
  H^F_\Om=\Div(\nabla F(\nabla u))\,.
\end{equation}
Indeed by $\nu_\Om=\nabla u/|\nabla u|$ and the zero-homogeneity of $\nabla F$, we have
\[
H^F_\Om=\Div^{\pa\Om}(\nabla F(\nu_\Om))=\Div^{\pa\Om}(\nabla F(\nabla u))=\Div(\nabla F(\nabla u))\,,
\]
provided that
\[
\nabla u\cdot\nabla\big(\nabla F(\nabla u)\big)[\nabla u]=0\,.
\]
To prove this last identity, we denote by superscripts components and by subscripts partial derivatives, and compute
\begin{eqnarray*}
\big(\nabla(\nabla F(\nabla u))\big)^{ij}&=&(F_{\xi_j}(\nabla u))_{x_i}=\sum_k F_{\xi_j\xi_k}(\nabla u)u_{x_kx_i}\,,
\\
\big(\nabla\big(\nabla F(\nabla u)\big)[\nabla u]\big)^j&=&\sum_{ik}F_{\xi_j\xi_k}(\nabla u)u_{x_kx_i}u_{x_i}\,,
\\
\nabla u\cdot\nabla\big(\nabla F(\nabla u)\big)[\nabla u]&=&\sum_{ijk}F_{\xi_j\xi_k}(\nabla u)u_{x_kx_i}u_{x_i}u_{x_j}
\\
&=&\sum_{ik}u_{x_kx_i}u_{x_i}\sum_jF_{\xi_j\xi_k}(\nabla u)u_{x_j}\,.
\end{eqnarray*}
The last sum over $j$ is equal to zero because $\nabla^2F(\nu)[\nu]=0$ for every $\nu\in\S^n$. Indeed, $F_{\xi_k}(t\,\nu)=F_{\xi_k}(\nu)$ for every $t>0$ and $\nu\in\S^n$.
This proves \eqref{wx1}. Then,
\begin{eqnarray*}
\Delta_Fu&=&\Div(F(\nabla u)\nabla F(\nabla u))=\sum_i\big(F(\nabla u)F_{\xi_i}(\nabla u)\big)_{x_i}
  \\
  &=&\sum_{ij}F_{\xi_j}(\nabla u)F_{\xi_i}(\nabla u)u_{x_jx_i}+F(\nabla u)F_{\xi_i\xi_j}(\nabla u)u_{x_jx_i}
  \\
  &=&\nabla^2u[\nabla F(\nabla u),\nabla F(\nabla u)]+F(\nabla u)\Div(\nabla F(\nabla u))\,,
\end{eqnarray*}
and thus \eqref{eqn: ARMA id} holds thanks to \eqref{wx1}. We now exploit \eqref{eqn: ARMA id}  in the proof of the following anisotropic version of Reilly's identity.

\begin{proposition}[Anisotropic Reilly's identity]\label{proposition reilly} If $\Om$ is a bounded open set with smooth boundary, $F$ is an elliptic integrand and $u$ is the $F$-anisotropic torsion potential of $\Om$, then
\begin{equation}\label{reilly id anisop}
\int_\Omega (\Delta_F u)^2 -\trace\left((\nabla(\nabla_Fu))^2\right)= \int_{\pa \Om} H_\Om^F \, F(\nabla u)^2 F(\nu_\Om)\, d\H^n\,,
\end{equation}

\end{proposition}
In \eqref{reilly id anisop} we use notation introduced in \eqref{eqn: cahn}. Notice there are no regularity issues in \eqref{reilly id anisop} as $u\in H^2(\Om)$ and is smooth in a neighborhood of $\pa\Om$ (Proposition \ref{section torsion potential}). We first prove the following lemma:

\begin{lemma}\label{lemma bernstein}
  If $V\in H^1(\Om;\R^{n+1})$ is such that $\Div V=c_0$, then $S=(\Div V)V-\nabla V[V]\in L^1(\Om;\R^{n+1})$ with $\Div S\in L^1(\Om)$ given by
  \begin{equation}\label{eqn: Bern id}
  \Div\,S=(\Div V)^2-\trace((\nabla V)^2)\,.
  \end{equation}
\end{lemma}

\begin{proof}[Proof of Lemma \ref{lemma bernstein}]
  Since $\Div V$ is constant we immediately find $\Div((\Div V)V)=(\Div V)^2$. First, we assume that $V$ is smooth. If $\vphi\in C^\infty_c(\Om)$, then
  \[
  \int_\Om \nabla V[V]\cdot\nabla\vphi=\sum_{ij}\int_\Om V^i_{x_j}V^j \vphi_{x_i}=\sum_{ij}\int_\Om V^i_{x_j}\,(V^j\vphi)_{x_i}-\int_\Om\,\vphi\,\trace((\nabla V)^2)\,,
  \]
  where
  \begin{eqnarray*}
  \sum_{ij}\int_\Om  V^i_{x_j}\,(V^j\vphi)_{x_i}
  &=&-\sum_{ij}\int_\Om V^i\,(V^j\vphi)_{x_i x_j}
  =\sum_{ij}\int_\Om V^i_{x_i}\,(V^j\vphi)_{x_j}
  \\
  &=&c_0 \sum_j\int_\Om (V^j\vphi)_{x_j}=0
  \end{eqnarray*}
  as $V^j\vphi=0$ on $\pa\Om$. This proves that if $V\in H^1(\Om)\cap C^{\infty}(\Om)$, and $\Div V$ is constant, then
  \begin{equation}
    \label{b}
      \int_\Om \nabla V[V]\cdot\nabla\vphi=-\int_\Om\vphi\,\trace((\nabla V)^2)\,,\qquad\forall \vphi\in C^\infty_c(\Om)\,.
  \end{equation}
  Now let $V\in H^1(\Om)$ with $\Div V=c_0$. Fix $\vphi\in C^\infty_c(\Om)$ and consider an $\e$-regularization $V_\e=V\star\rho_\e$ of $V$, so that $V_\e$ is smooth on $\Om_\e=\{x\in\Om:\dist(x,\pa\Om)>\e\}$. Since $\Div V_\e=(\Div V)\star\rho_\e$ is constant on $\Om_\e$, we can apply \eqref{b} to $V_\e$ if $\e$ is small enough to have $\spt\vphi\cc\Om_\e$. Since $V_\e\to V$ in $H^1$ on an open neighborhood of $\spt\vphi$, we can pass to the limit as $\e\to 0$ and deduce that \eqref{b} holds without the smoothness assumption on $V$.
\end{proof}

\begin{proof}[Proof of Proposition \ref{proposition reilly}] Since $u\in H^2(\Om)$, we apply Lemma \ref{lemma bernstein} with $V=\nabla_Fu$ to find
\[
\int_\Omega (\Delta_F u)^2 -\trace\left((\nabla(\nabla_Fu))^2\right)=\int_\Om\,\Div\big(\Delta_Fu\nabla_Fu-\nabla(\nabla_Fu)[\nabla_Fu]\big)\,.
\]
Since $u$ is smooth in a neighborhood of $\pa\Om$ with $\nu_\Om=\nabla u/|\nabla u|$ we can apply the divergence theorem to get
\[
\int_\Omega (\Delta_F u)^2 -\trace\left((\nabla(\nabla_Fu))^2\right)=\int_{\pa\Om}\,\big(\nabla(\nabla_Fu)[\nabla_Fu]-\Delta_Fu\nabla_Fu\big)\cdot\frac{\nabla u}{|\nabla u|}\,.
\]
 By
$\na \frac{F^2}{2}(\nabla u)\cdot\nabla u=F(\nabla u)^2$
 and \eqref{eqn: ARMA id}, we have
\begin{eqnarray*}
  -\int_{\pa\Om}\Delta_Fu\,\nabla_Fu\cdot\frac{\nabla u}{|\nabla u|}=
  \int_{\pa\Om}H_\Om^F\,F(\nabla u)^2 F(\nu_{\Om})
  -\nabla^2u[\nabla F(\nabla u),\nabla F(\nabla u)]\,\frac{F(\nabla u)^2}{|\nabla u|}\,.
\end{eqnarray*}
Thus, in order to complete the proof, it suffices to show that
\begin{equation}\label{eqn: equal terms}
\nabla(\na_F u)[\na_F u]\cdot \nabla u=F(\na u)^2\,\nabla^2u [ \na F(\na  u) , \na F(\na u)]\,.
\end{equation}
We have that
\[
[\nabla(\na_F u)]^{ij}=(F(\nabla u)F_{\xi_i}(\nabla u))_{x_j}=\sum_k F(\nabla u)F_{\xi_i\xi_k}(\nabla u)\,u_{x_kx_j}
+F_{\xi_i}(\nabla u)F_{\xi_k}(\nabla u)u_{x_jx_k}\,.
\]
Hence
\begin{eqnarray*}
\nabla(\na_F u)[\nabla_Fu]\cdot\nabla u&=&\sum_i \big(\nabla(\na_F u)[\nabla_Fu]\big)^iu_{x_i}
\\
&=&F(\nabla u)\sum_{ij}[\nabla(\na_F u)]^{ij}F_{\xi_j}(\nabla u)u_{x_i}
\\
&=&F(\nabla u)^2\sum_{kj} F_{\xi_j}(\nabla u)u_{x_kx_j}\sum_iF_{\xi_i\xi_k}(\nabla u)\,u_{x_i}
\\
&&+F(\nabla u) \sum_{kj} F_{\xi_k}(\nabla u)F_{\xi_j}(\nabla u)\,u_{x_jx_k}\,\sum_i F_{\xi_i}(\nabla u)\,u_{x_i}\,.
\end{eqnarray*}
Again, $\na F(\nu)[\nu]=0$ by homogeneity, so
\[
\sum_iF_{\xi_i\xi_k}(\nabla u)\,u_{x_i}=\nabla^2F(\nabla u)[\nabla u,e_k]=0\,.
\]
At the same time, $\sum_i F_{\xi_i}(\nabla u)\,u_{x_i}=\nabla F(\nabla u)\cdot\nabla u=F(\na u)$, so \eqref{eqn: equal terms} is proven.
\end{proof}

\subsection{Universal Lipschitz estimate}\label{section lipschitz estimate} We now turn to the proof of the following gradient bound for $u$, which is notably independent of the ellipticity constants of $F$.

\begin{proposition}\label{proposition lipschitz estimate}
  If $F$ is an elliptic integrand, $\Om$ is a bounded open smooth set with $H^F_\Om>0$ on $\pa\Om$ and $u$ is the $F$-anisotropic torsion potential of $\Om$, then
  \begin{equation}
    \label{lipschitz estimate}
    \sup_{\Om}|\nabla u|\le \frac1{m_F\,H_{\Om,\inf}^F}
  \end{equation}
  where $H^F_{\Om,\inf}$ denotes the infimum of $H^F_\Om$ over $\pa\Om$.
\end{proposition}

\begin{proof} Recall from Proposition~\ref{proposition signed distance} that the $F$-anisotropic signed distance function $\g$ of $\Om$, defined in \eqref{anisotropic signed distance}, is smooth in a neighborhood of $\pa\Om$ with $\g<0$ in $\Om$ and
\[
\nabla\g(y)=\frac{\nu_\Om(y)}{F(\nu_\Om(y))}\qquad\forall y\in\pa\Om\,.
\]
Hence, $F(\nabla\g)=1$ on $\pa\Om$, and so
\[
\Delta_F\g=\Div(F(\nabla\g)\nabla F(\nabla\g))=\Div(\nabla F(\nu_\Om))=H^F_\Om\qquad\mbox{on $\pa\Om$}\,
\]
by the zero-homogeneity of $\nabla F$ and \eqref{wx1}.

\medskip

\noindent {\it Step one}: We show that $\g$ satisfies $\Delta_F\g\ge H^F_{\Om,\inf}$ in the viscosity sense in $\Om$. Indeed, let $P$ denote a second order polynomial touching $\g$ from above at some $x\in\Om$, i.e., for some $r>0$, $P\ge \g$ on $B_r(x)\subset\Om$ with $P(x)=\g(x)$. We want to prove that
\[
\Delta_FP(x)\ge H^F_{\Om,\inf}\,.
\]
Let $y\in p(x)$, where $p$ is the projection map as defined in \eqref{eqn: projection}, and note that $x_t=x+t\,(y-x)\in\Om$ for all $t\in(0,1)$. In particular, let $r_t\in(0,r)$ be such that $B_{r_t}(x_t)\subset\Om$. We claim that the second order polynomial $Q_t$ defined by
\[
Q_t(z)=P(z-x_t+x)+F_*(x_t-x)\,,\qquad z\in\R^{n+1}
\]
is such that
\begin{equation}
  \label{Qt}
  Q_t(x_t)=\g(x_t)\,, \qquad \qquad Q_t(z)\ge \g(z)\qquad\forall z\in B_{r_t}(x_t)\,.
\end{equation}
Indeed,
\[
Q_t(x_t)=P(x)+F_*(x_t-x)=\g(x)+t\,F_*(y-x)=(1-t)\g(x)=\g(x_t)\,,
\]
while if $z\in B_{r_t}(x_t)$ then $z-x_t+x\in B_r(x)\subset\Om$, and thus
\begin{eqnarray*}
Q_t(z)&\ge&\g(z-x_t+x)+F_*(x_t-x)=-\inf\big\{F_*(w-(z-x_t+x)):w\in\Om^c\big\}+F_*(x_t-x)
  \\
  &\ge&-\inf\big\{F_*(w-z):w\in\Om^c\big\}-F_*(x_t-x)+F_*(x_t-x)=\g(z)
\end{eqnarray*}
as $z\in\Om$. Now, if $t$ is close enough to $1$, then $x_t$ lies in the neighborhood $N$ of $\pa\Om$ in which $\g$ is smooth. Hence, for $t$ close enough to $1$, \eqref{Qt} implies that
\[
\Delta_FP(x)=\Delta_FQ_t(x_t)\ge\Delta_F\g(x_t)\,.
\]
Letting $t\to 1^-$, we find $\Delta_FP(x)\ge\Delta_F\g(y)=H^\Om_F(y)\ge H^\Om_{F,\inf}$, as required.

\bigskip

\noindent {\it Step two}: We claim that
\begin{equation}
  \label{gamma minore u}
  \g(x)\le H^F_{\Om,\inf}\,u(x)\qquad\forall x\in\Om\,.
\end{equation}
Both functions are equal to zero on $\pa\Om$ and are strictly negative in $\Om$, so the claim follows by showing that, for every $\a>0$,
\[
v(x)=(H^F_{\Om,\inf}-\a)u(x)-\g(x)\ge 0\qquad\forall x\in\Om\,.
\]
In turn, since $v$ is continuous in $\ov{\Om}$ and $v=0$ on $\pa\Om$, it suffices to show that $v$ has no interior minimum points in $\Om$.

Arguing by contradiction, let $x_0\in\Om$ be a local minimum point of $v$ in $\Om$. If $\nabla u(x_0)\ne 0$, then $u$ is smooth nearby $x_0$. Moreover, since $x_0$ is a local minimum of $v$, we have that
\[
(H^F_{\Om,\inf}-\a)u-[(H^F_{\Om,\inf}-\a)u(x_0)-\g(x_0)]
\]
touches $\g$ from above at $x_0$, and thus, by Step one,
\[
H^F_{\Om,\inf}\le\Delta_F[(H^F_{\Om,\inf}-\a)u](x_0)=(H^F_{\Om,\inf}-\a)\,\Delta_Fu(x_0)=H^F_{\Om,\inf}-\a\,,
\]
a contradiction to $\a>0$. This implies that $\nabla u(x_0)=0$. Hence $\g$ cannot be differentiable at $x_0$: otherwise $\nabla v(x_0)=0$ would imply $\nabla\g(x_0)=0$, whereas $\nabla\g=(\nu_\Om\circ p)/F(\nu_\Om\circ p)\ne 0$ at every differentiability point of $\g$.

We are thus left to consider the possibility of a local minimum point $x_0\in\Om$ of $v$, where $\nabla u(x_0)=0$ and $\g$ is not differentiable at $x_0$. By local minimality, for every $\nu\in\S^n$ and $t>0$ small enough, we have
\[
(H^F_{\Om,\inf}-\a)\frac{u(x_0+t\,\nu)-u(x_0)}t-\frac{\g(x_0+t\,\nu)-\g(x_0)}t\ge0\,.
\]
So $\nabla u(x_0)=0$ implies
\[
\liminf_{t\to 0^+}\frac{\g(x_0+t\,\nu)-\g(x_0)}t\le 0\,.
\]
Since $\g$ is not differentiable at $x_0$, there exists $\nu_0\in\S^n$ such that
\[
\liminf_{t\to 0^+}\frac{\g(x_0+t\,\nu_0)-\g(x_0)}t< 0\,,
\]
which in turn implies
\begin{equation}
  \label{minus infinity}
  \liminf_{t\to 0^+}\frac{\g(x_0+t\,\nu_0)+\g(x_0-t\nu_0)-2\g(x_0)}{t^2}=-\infty\,.
\end{equation}
We will now obtain a contradiction to \eqref{minus infinity} by proving that for every $x\in\Om$ there exists a second order polynomial $P$ touching $\g$ from below at $x$.

Indeed, since $\g$ is smooth in a neighborhood $N$ of $\pa\Om$, there exists $C>0$ such that for every $z\in N$ there exists a second order polynomial $P_z$ with the properties
\[
P_z(z)=\g(z)\,,\qquad\nabla^2P_z\ge-C\,\Id\,,\qquad\mbox{$\g(w)\ge P_z(w)$ for every $w$ near $z$}\,.
\]
Choose $x\in\Om$ with $B_r(x)\subset\Om$, let $y\in p(x)$, and set $x_t=x+t(y-x)$. For $t$ sufficiently close to $1$, $x_t\in N$ and we can find $r_t\in(0,r)$ such that $B_{r_t}(x_t)\subset\Om$ and, setting $P_t=P_{x_t}$,
\[
P_t(x_t)=\g(x_t)\,,\qquad\nabla^2P_t\ge-C\,\Id\,,\qquad\mbox{$\g(w)\ge P_t(w)$ for every $w\in B_{r_t}(x_t)$}\,.
\]
Set
\[
P(w)=P_t(w-(x-x_t))+\g(x)-\g(x_t)\,.
\]
Clearly $P(x)=\g(x)$. If $w\in B_{r_t}(x)\subset\Om$, then $w-(x-x_t)\in B_{r_t}(x_t)$ and thus
\begin{align*}
P(w)&\le\g(w-(x-x_t))+\g(x)-\g(x_t) \\
&=\g(w-(x-x_t))-t\g(x) =\g(w-(x-x_t))-t\,F_*(y-x)\,.
\end{align*}
Now, since $w-(x-x_t)\in B_{r_t}(x_t)\subset\Om$, by definition of $\g$ there exists $\bar w\in\pa\Om$ such that
\[
\g(w-(x-x_t))=-F_*(\bar w-(w-(x-x_t)))
\]
so that, by the subadditivity of $F_*$,
\begin{align*}
P(w)&\le-F_*(\bar w-(w-(x-x_t)))-F_*(t(y-x))\\
&= -F_*(\bar w-(w+t(y-x)))-F_*(t(y-x))\\
&\le-F_*(\bar w-w)\le\g(w)\,,
\end{align*}
again by definition of $\g$ in $\Om$, and thanks to $w\in\Om$. This completes the proof of step two.

\bigskip

\noindent {\it Step three}: We claim that if $u$ is the $F$-anisotropic torsion potential of $\Om$, then
\begin{equation}
  \label{gt}
  \sup_\Om|\nabla u|\le\frac1{m_F\,H^F_{\Om,\inf}}\,.
\end{equation}
First, let us  pick $x\in\Om$ and $y\in\pa\Om$. By \eqref{gamma minore u}, $-u(x)\le-\g(x)/H^F_{\Om,\inf}$, thus
\[
\frac{|u(x)-u(y)|}{|x-y|}=\frac{-u(x)}{|x-y|}\le\frac1{H^F_{\Om,\inf}}\,\frac{-\g(x)}{|x-y|}\,,
\]
so, first using $|x-y|/m_F\ge F_*(y-x)=-\g(x)$ and then letting $x\to y$, we have proven
\begin{equation}
  \label{gt0}
  \sup_{\pa\Om}|\nabla u|\le\frac1{m_F\,H^F_{\Om,\inf}}\,.
\end{equation}
We now prove that $|\nabla u|$ (which is a H\"older continuous function on $\ov{\Om}$) achieves its maximum on $\pa\Om$.
Recall that $u$ is smooth and solves $1=\Delta_Fu$ pointwise on $\Om\setminus\C$, with
\[
\Delta_Fu=\sum_{ij}\Big(F(\nabla u)\,F_{\xi_i\xi_j}(\nabla u)+F_{\xi_i}(\nabla u)F_{\xi_j}(\nabla u)\Big)\,u_{x_ix_j}\,.
\]
Since the operator $\Delta_F$ is smooth on $\Om \setminus \C$, by \cite[Theorem 15.1]{GT}, the maximum of $|\nabla u|$ is attained on $\pa (\Om \setminus \C)$. Then, since $|\na u| = 0$ on $\C$,
\[
\sup_\Om|\nabla u| \le \sup_{\pa\Om}|\nabla u| \le \frac1{m_F\,H^F_{\Om,\inf}}\,.
\]
\end{proof}
\subsection{The anisotropic Heintze-Karcher and bounds on the torsion potential}\label{section hk with bounds} In this section we give a proof of the anisotropic Heintze-Karcher inequality \eqref{hk inq anisotropic} that uses the properties of the anisotropic torsion potential, in the spirit of Ros' argument \cite{Ros}; see Proposition \ref{proposition hessian bound} below. We recall from the introduction the definitions, for an open set with smooth boundary $\Om$,
\begin{eqnarray*}
\eta_F(\Om)&=&1-\frac{(n+1)|\Om|}{\int_{\pa\Om}\big(n\,F(\nu_\Om)/H^F_\Om\big)}\,,\qquad\mbox{if $H^F_\Om>0$}\,,
\\
\delta_F(\Om)&=&\Big(\frac1{\F(\Om)}\,\int_{\pa\Om}\Big|\frac{H_{\Om}^F}{H_{\Om}^{F,0}}-1\Big|^2F(\nu_{\Om})\Big)^{1/2}\,,
\end{eqnarray*}
where $H^{F,0}_\Om=n\,\F(\Om)/(n+1)|\Om|$. We have the following relation between $\eta_F$ and $\de_F:$
%

\begin{lemma}\label{lemma magenta} If $\k>0$ and $\Om$ is a smooth bounded open set with $H_{\Om}^{F}\ge \k H_{\Om}^{F,0}$ on $\pa\Om$, then
\[
\eta_F(\Om)\le \frac{1}{\k\,\F(\Om)}\int_{\pa\Om}\Big|\frac{H^F_\Om}{H^{F,0}_\Om}-1\Big|\,F(\nu_\Om)\,.
\]
In particular,
\[
\eta_F(\Om)\le\de_F(\Om)/\k.
\]
\end{lemma}
\begin{proof} Since $(n+1)|\Om|=\int_{\Om}n\,F(\nu_\Om)/H^{F,0}_\Om$, we find
\begin{eqnarray*}
\eta(\Om)
&=&
\frac{1}{\int_{\pa \Om}(n\,F(\nu_\Om)/H^F_\Om)}\,
\Big(\int_{\pa\Om}\frac{n\,F(\nu_\Om)}{H_{\Om}^{F}}-\int_{\pa\Om}\frac{n\,F(\nu_\Om)}{H_{\Om}^{F,0}}\Big)
\\
&=&
\frac{1}{\int_{\pa\Om}(\,F(\nu_\Om)/H^F_\Om)}
\,\int_{\pa\Om}\Big(1-\frac{H_{\Om}^{F}}{H_{\Om}^{F,0}}\Big)\,\frac{F(\nu_\Om)}{H_{\Om}^{F}}
\\
&\le&
\frac{1}{\k\,H^{F,0}_{\Om}\,\int_{\pa\Om}(\,F(\nu_\Om)/H^F_\Om)}
\,\int_{\pa\Om}\Big|1-\frac{H_{\Om}^{F}}{H_{\Om}^{F,0}}\Big|\,F(\nu_\Om)
\\
&=&
\frac{(n+1)|\Om|}{\int_{\pa\Om}(n\,F(\nu_\Om)/H^F_\Om)}
\,\frac{1}{\k\,\F(\Om)}\,\int_{\pa\Om}\Big|1-\frac{H_{\Om}^{F}}{H_{\Om}^{F,0}}\Big|\,F(\nu_\Om)
\end{eqnarray*}
where the first factor is less than $1$ exactly by the anisotropic Heintze-Karcher inequality. The proposition now follows from H\"{o}lder's inequality.
\end{proof}

\begin{proposition}\label{proposition hessian bound}
  If $F\in C^\infty(\S^n)$ is an elliptic integrand, $\Om$ is a bounded open smooth set with $H^F_\Om>0$, and $u$ is the $F$-torsion potential of $\Om$, then
  \begin{eqnarray}
    \label{hk anisotropic plus}
    \frac{|\Om|}{n+1}\Big(\int_{\pa\Om}\frac{n\,F(\nu_\Om)}{H^F_\Om}-(n+1)|\Om|\Big)&=&
    \int_{\pa\Om}\frac{F(\nu_\Om)}{H^F_\Om}\Big(\int_\Om\trace((\nabla(\nabla_Fu))^2)-\frac{(\Delta_Fu)^2}{n+1}\Big)
    \\\nonumber
    &+&\int_{\pa\Om}\frac{F(\nu_\Om)}{H^F_\Om}\int_{\pa\Om}H^F_\Om F(\nabla u)^2 F(\nu_\Om)-\Big(\int_{\pa\Om}F(\nabla u) F(\nu_\Om)\Big)^2\,.
  \end{eqnarray}
  Moreover: \begin{itemize}
  \item[(a)] The right-hand side of \eqref{hk anisotropic plus} is non-negative, so \eqref{hk anisotropic plus} implies the anisotropic Heintze-Karcher inequality \eqref{hk inq anisotropic};
  \item[(b)] If the right-hand side is equal to zero and $\Om$ is connected, then $\Om=x+r\,K_F$ for some $r>0$;
  \item[(c)] We have
  \begin{eqnarray}
    \label{estimate 1}
      C(n)\frac{\Lambda_*}{\lambda_*}\,|\Om|\,\eta_F(\Om)&\ge&\int_\Om\Big|\nabla(\nabla_Fu)-\frac{\Id}{n+1}\Big|^2\,,
      \\
      \label{estimate 1b}
      \frac{C(n)}{\lambda_*^2}\,|\Om|\,\eta_F(\Om)&\ge&\int_\Om\Big|\na^2u-\frac{\na^2(F_*^2/2)(\na F(\na u))}{n+1}\Big|^2\,,
  \end{eqnarray}
  and, if $\de_F(\Om)<1/2$ and $H^F_\Om\ge\k H^{F,0}_\Om$ for some $\k>0$ on $\pa\Om$, then
\begin{equation} \label{estimate 2}
  \frac{C(n)}{\k}\,\frac{|\Om|^{(n+2)/(n+1)}}{|K_F|^{1/(n+1)}}\,
  \Big\{\eta_F(\Om)+\frac{\de_F(\Om)^2}\k\Big\}
  \ge\int_{\pa\Om}\Big|\frac{(n/H^{F,0}_\Om)}{n+1}-F(\nabla u)\Big|^2\,F(\nu_\Om)\,.
  \end{equation}
  \end{itemize}
\end{proposition}

\begin{proof} {\it Step one}: We prove \eqref{hk anisotropic plus}, following Ros' argument in \cite{Ros}. By the divergence theorem, since $\nu_\Om=\nabla u/|\nabla u|$ on $\pa\Om$, we have
  \begin{eqnarray*}
  |\Om|^2&=&\Big(\int_\Om\Delta_Fu\Big)^2=\Big(\int_{\pa\Om} F(\nabla u)\nabla F(\nabla u)\cdot\nu_\Om\Big)^2
  =\Big(\int_{\pa\Om} \frac{F(\nabla u)^2}{|\nabla u|}\Big)^2=\Big(\int_{\pa\Om} F(\nabla u)\,F(\nu_\Om)\Big)^2
  \\
  &\le&\int_{\pa\Om}H^F_\Om\,F(\nabla u)^2\,F(\nu_\Om)\,\int_{\pa\Om}\frac{F(\nu_\Om)}{H^F_\Om}\,.
  \end{eqnarray*}
  By the anisotropic Reilly's identity,
  \begin{eqnarray*}
  \int_{\pa \Om} H_\Om^F \, F(\nabla u)^2 F(\nu_\Om)\, d\H^n&=&\int_\Omega (\Delta_F u)^2 -\trace\left((\nabla(\nabla_Fu))^2\right)
  \\
  &=&\frac{n}{n+1}\int_\Omega (\Delta_F u)^2 +\int_\Omega \frac{(\Delta_F u)^2}{n+1} -\trace\left((\nabla(\nabla_Fu))^2\right)
  \\
  &\le&\frac{n}{n+1}\int_\Omega (\Delta_F u)^2 =\frac{n|\Om|}{n+1}\,,
  \end{eqnarray*}
where the last inequality follows from the following linear algebra consideration. Set $A=\nabla^2(F^2/2)(\nabla u)$ and $B=\nabla^2 u$, so $\nabla(\nabla_Fu)=AB$ and therefore
  \[
  \trace((\nabla(\nabla_Fu))^2)=\trace((AB)^2)\,.
  \]
  Since $A$ is positive definite and symmetric on $\Om\setminus\C$, $A^{1/2}$ and $A^{-1/2}$ exist and
  \[
  A^{1/2}BA^{1/2}=A^{-1/2}(AB)A^{1/2}
  \]
  is symmetric. Hence, there exist an orthogonal matrix $O$ and a diagonal matrix $L$ such that
  \begin{equation}
    \label{smart}
      A^{1/2}BA^{1/2}=O^{-1}LO\,.
  \end{equation}
  So,
  \begin{equation}\label{eqn: matrices}
  AB=A^{1/2}O^{-1}LOA^{-1/2}\,,\qquad (AB)^2=A^{1/2}O^{-1}L^2OA^{-1/2}
  \end{equation}
 and
 \begin{equation}\label{eqn: matrices2}
 B - A^{-1}= A^{-1/2}O^{-1}(L-\Id)OA^{-1/2}\,.
 \end{equation}
 Because the trace operator is invariant under conjugation,
  \[
  \trace(AB)=\trace(L)\,,\qquad\trace((AB)^2)=\trace(L^2)\,.
  \]
  By H\"older's inequality, $(n+1)\,\trace(L^2)\ge\trace(L)^2$ so that
  \[
  \trace((\nabla(\nabla_F u))^2)-\frac{(\Delta_Fu)^2}{n+1}=\trace(L^2)-\frac{(\trace L)^2}{n+1}\ge0\,,
  \]
We have now shown the anisotropic Heintze-Karcher inequality and the identity \eqref{hk anisotropic plus} for the anisotropic torsion potential $u$.

  \medskip

  \noindent {\it Step two}: We prove \eqref{estimate 1}, \eqref{estimate 1b}, and \eqref{estimate 2}. Let us first notice that if $v$ and $w$ belong to a Hilbert space with norm $|\cdot|$ and $v\cdot w\ge0$ then
  \begin{eqnarray}\nonumber
    |v|^2|w|^2-(v\cdot w)^2&\ge&|v|\,|w|\,\Big(|v||w|-(v\cdot w)\Big)=\frac{|v|\,|w|}2\,\bigg|\sqrt{\frac{|w|}{|v|}}v-
    \sqrt{\frac{|v|}{|w|}}w\bigg|^2
    \\\label{hilbert}
    &=&\frac{|v|^2}2\,\bigg|\frac{|w|}{|v|}v-w\bigg|^2\,.
  \end{eqnarray}
  Given $(\l_i)_{i=1}^N$ such that $\sum_i\l_i=1$, we can apply \eqref{hilbert} to $w=\sum_i\l_i\,e_i$ and $v=\sum_ie_i$ to obtain
  \begin{eqnarray*}
  N\,\sum_i\l_i^2-(\sum_i\l_i)^2&\ge&\frac{N}2\,\Big|\frac{(\sum_i\l_i^2)^{1/2}}{N}\,\sum_ie_i-\sum_i\l_ie_i\Big|^2
  \\
  &\ge&\frac{N}2\,\inf_{c\in\R}\sum_i(c-\l_i)^2=\frac{N}2\sum_i\Big(\frac{\sum_j\l_j}N-\l_i\Big)^2
  =\frac{N}2\sum_i\Big(\frac1N-\l_i\Big)^2\,.
  \end{eqnarray*}
  Thanks to \eqref{elliptic integrand GF},
  \[
  |A^{1/2}|\le \sqrt{(n+1)\,\Lambda_*}\qquad |A^{-1/2}|\le\sqrt{\frac{n+1}{\l_*}}\,.
  \]
  By \eqref{eqn: matrices}, if we set $N=n+1$ and denote by $\l_i$ the eigenvalues of $L$, then
  \begin{align*}
  \Big|\nabla(\nabla_F u)-\frac{\Id}{n+1}\Big|^2 &\le C(n)\frac{\Lambda_*}{\lambda_*}\,\Big|L-\frac{\Id}{n+1}\Big|^2\\
  &
  \le C(n)\frac{\Lambda_*}{\lambda_*}
  \sum_i\Big(\frac1{n+1}-\l_i\Big)^2\le C(n)\frac{\Lambda_*}{\lambda_*}\, \Big(\trace(L^2)-\frac{\trace(L)^2}{n+1}\Big)\,.
  \end{align*}
  Hence, by \eqref{hk anisotropic plus},
  \[
  \frac{n\,|\Om|}{n+1}\,\eta_F(\Om)\ge \int_\Om\trace((\nabla(\nabla_Fu))^2)-\frac{(\Delta_Fu)^2}{n+1}
  \ge\frac1{C(n)}\frac{\l_*}{\Lambda_*}\int_\Om\Big|\nabla(\nabla_F u)-\frac{\Id}{n+1}\Big|^2\,.
  \]
 This proves \eqref{estimate 1}. In the same way, using \eqref{eqn: matrices2} and recalling \eqref{eqn: hessian inverses} to express $A^{-1}$, we have
  \begin{align*}
  \Big| \na^2 u - \frac{\na^2 (F^2_*/2)(\na F(\na u))}{n+1}\Big|^2	
  &\le \frac{C(n)}{\lambda_*^2}\,\Big|L-\frac{\Id}{n+1}\Big|^2\\
&  \le \frac{C(n)}{\lambda_*^2}
  \sum_i\Big(\frac1{n+1}-\l_i\Big)^2
  \le\frac{C(n)}{\lambda_*^2}\Big(\trace(L^2)-\frac{\trace(L)^2}{n+1}\Big)\,.
  \end{align*}
  So, again by \eqref{hk anisotropic plus},
   \[
  \frac{n\,|\Om|}{n+1}\,\eta_F(\Om)\ge \int_\Om\trace((\nabla(\nabla_Fu))^2)-\frac{(\Delta_Fu)^2}{n+1}
  \ge\frac{\l_*^2}{C(n)}\int_\Om\Big|\na^2 u -\frac{\na^2(F_*^2/2)(\na F(\na u))}{n+1}\Big|^2\,,
  \]
proving \eqref{estimate 1b}.

Similarly, \eqref{hk anisotropic plus} implies
  \[
  \frac{|\Om|}{n+1}\,\eta_F(\Om)\,\int_{\pa\Om}\frac{n\,F(\nu_\Om)}{H^F_\Om}\ge |v|^2|w|^2-(v\cdot w)^2
  \]
  where $v=(H^F_\Om)^{-1/2}$ and $w=(H^F_\Om)^{1/2}F(\nabla u)$ are seen as vectors of $L^2(F(\nu_\Om)\H^n\llcorner\pa\Om)$. Writing down \eqref{hilbert} we find
  \[
  \frac{|\Om|}{n+1}\,\eta_F(\Om)\,\int_{\pa\Om}\frac{n\,F(\nu_\Om)}{H^F_\Om}
  \ge\frac12\int_{\pa\Om}\frac{F(\nu_\Om)}{H^F_\Om}\int_{\pa\Om}\Big|\frac{|w|}{|v|}\frac{1}{(H_\Om^F)^{1/2}}-(H_\Om^F)^{1/2}F(\nabla u)\Big|^2\,F(\nu_\Om)\,.
  \]
  After simplifying we have,
  \[
  C(n)\,|\Om|\,\eta_F(\Om)\ge\int_{\pa\Om}\Big|\frac{|w|}{|v|}\frac{1}{H_\Om^F}-F(\nabla u)\Big|^2\,H^F_\Om\,F(\nu_\Om)
  \]
  and hence
  \[
  C(n)\,
  \Big\{|\Om|\,\eta_F(\Om)
  +\frac{|w|^2}{|v|^2}\int_{\pa\Om}\Big|\frac1{H^{F,0}_\Om}-\frac1{H^F_\Om}\Big|^2\,H^F_\Om\,F(\nu_\Om)
  \Big\}
  \ge\int_{\pa\Om}\Big|\frac{|w|}{|v|}\frac{1}{H_\Om^{F,0}}-F(\nabla u)\Big|^2\,H^F_\Om\,F(\nu_\Om)\,.
  \]
  Thanks to \eqref{reilly id anisop} we have $|w|^2\le|\Om|$, so that by $1/H^F_\Om\le1/\k\,H^{F,0}_\Om$,
  \begin{eqnarray*}
  \frac{|w|^2}{|v|^2}\int_{\pa\Om}\Big|\frac1{H^{F,0}_\Om}-\frac1{H^F_\Om}\Big|^2\,H^F_\Om\,F(\nu_\Om)
  &\le&
  \frac{|\Om|}{\int_{\pa\Om}F(\nu_\Om)/H^F_\Om}\int_{\pa\Om}\Big|\frac{H^F_\Om}{H^{F,0}_\Om}-1\Big|^2\,\frac{F(\nu_\Om)}{H^F_\Om}
  \\
  &\le&\frac{|\Om|}{\k\,H^{F,0}_\Om\,\int_{\pa\Om}F(\nu_\Om)/H_\Om^F}\int_{\pa\Om}\Big|\frac{H^F_\Om}{H^{F,0}_\Om}-1\Big|^2\,F(\nu_\Om)
  \\
  &=&
  \frac{|\Om|}\k\,\frac{(n+1)\,|\Om|}{\int_{\pa\Om}n\,F(\nu_\Om)/H_\Om^F}\,\de_F(\Om)^2\le  \frac{|\Om|}\k\,\de_F(\Om)^2
  \end{eqnarray*}
  where in the last step we have used the anisotropic Heintze-Karcher inequality. Thus,
  \begin{equation}
    \label{newark1}
      C(n)\,|\Om|\,\Big\{\eta_F(\Om)+\frac{\de_F(\Om)^2}\k\Big\}
  \ge\int_{\pa\Om}\Big|\frac{|w|}{|v|}\frac{1}{H_\Om^{F,0}}-F(\nabla u)\Big|^2\,H^F_\Om\,F(\nu_\Om)\,.
  \end{equation}
  Exploiting again $H_\Om^{F}\ge \kappa H_\Om^{F,0}$
  \begin{eqnarray}\label{newark2}
        \int_{\pa\Om}\Big|\frac{|w|}{|v|}\frac{1}{H_\Om^{F,0}}-F(\nabla u)\Big|^2\,H^F_\Om\,F(\nu_\Om)
    \ge
    \kappa H^{F,0}_\Om\,\inf_{c\in\R}\int_{\pa\Om}\Big|c-F(\nabla u)\Big|^2\,F(\nu_\Om)\,.
  \end{eqnarray}
  By the divergence theorem, the optimal $c$ is given by the average
  \[
  c=\frac{\int_{\pa\Om}F(\nabla u)F(\nu_\Om)}{\int_{\pa\Om}F(\nu_\Om)}
  =
  \frac{\int_{\Om}\Delta_Fu}{\F(\Om)}=\frac{|\Om|}{\F(\Om)}=\frac{(n/H^{F,0}_\Om)}{n+1}\,.
  \]
  Thus by combining \eqref{newark1}, \eqref{newark2}, and the last identity, we have
  \begin{equation}
    \label{newark5}
    C(n)\,\frac{|\Om|}{\k\,H^{F,0}_\Om}\,\Big\{\eta_F(\Om)+\frac{\de_F(\Om)^2}\k\Big\}
        \ge \int_{\pa\Om}\Big|\frac{n/H^{F,0}_\Om}{n+1}-F(\nabla u)\Big|^2\,F(\nu_\Om)\,,
  \end{equation}
  where by the Wulff inequality \eqref{wulff inequality},
    \[
  \frac{|\Om|}{H^{F,0}_\Om}\le C(n)\,\frac{|\Om|^2}{\F(\Om)}\le C(n)\,\frac{|\Om|^{(n+2)/(n+1)}}{|K_F|^{1/(n+1)}}\,.
  \]
%
  We thus conclude the proof of \eqref{estimate 2}.

  \medskip

  \noindent {\it Step three}: We finally notice that if the right-hand side of \eqref{hk anisotropic plus} is equal to zero and $\Om$ is connected, then $\Om=x+r\,K_F$ for some $x\in\R^{n+1}$ and $r>0$. Indeed, in this case, \eqref{estimate 1} gives
  \[
  \nabla(\nabla_Fu)=\frac{\Id}{n+1}\qquad\mbox{on $\Om$}
  \]
  so that, for some $x_0\in\R^{n+1}$,
  \[
  \nabla_Fu(x)=\frac{x-x_0}{n+1}\qquad\forall x\in\Om\,.
  \]
  Since $\nabla_Fu(x)=\nabla(F^2/2)(\nabla u)$, by \eqref{inverse}
  \[
  \nabla u(x)=\nabla(F_*^2/2)\Big(\frac{x-x_0}{n+1}\Big)=\frac{\nabla (F_*)^2(x-x_0)}{2(n+1)}\qquad\forall x\in\Om\,.
  \]
  In particular, for some $c\in\R$ we have
  \[
  u(x)=c+\frac{F_*(x-x_0)^2}{2(n+1)}\qquad\forall x\in\Om\,,
  \]
  so that $u=0$ on $\pa\Om$ implies that $\pa\Om$ is a level set of $F_*(\cdot-x_0)$, as claimed.
  \end{proof}

\subsection{Proof of Theorem \ref{thm main}}\label{section proof final} {\it Step one}: We first recall our setting.  We consider a convex integrand $F$ that is the pointwise limit
  of a sequence $\{F_h\}_{h\in\N}$ of smooth elliptic integrands with
  \begin{equation}\label{hp main theorem xx}
  \begin{split}
    m\le F_h\le M&\qquad\mbox{on $\S^n$}\,,
    \\
    \l_h\,\Id\le \nabla^2 F_h(\nu)\le \Lambda_h\,\Id&\qquad\mbox{on $\nu^\perp$, $\forall\nu\in\S^n$}\,,
  \end{split}
  \end{equation}
  where $m$, $M$, $\l_h$, and $\Lambda_h$ are positive constants ($m$ and $M$ independent of $h$).
  Notice that, by convexity,
  \begin{equation}
    \label{convergence hp}
      \mbox{$F_h\to F$ uniformly on compact subsets of $\R^{n+1}$}\,.
  \end{equation}
  We also consider a sequence  $\{\Om_h\}_{h\in\N}$ of bounded open sets with smooth boundary with
  \begin{equation}
    \label{mc 123}
    H_{\Om_h}^{F_h, 0}=n\,,\qquad H^{F_h}_{\Om_h}\ge\k\,n\quad\mbox{on $\pa\Om_h$}
  \end{equation}
  and
  \begin{equation}
    \label{bounds x}
      \sup_{h\in\N}\diam(\Om_h)<\infty\,,\qquad \F_h(\Om_h)\le (L+\s)\,\F_h(K_{F_h})\,
  \end{equation}
  for some $\k>0$, $L\in\N$ and $\s\in(0,1)$. Moreover, we assume that
  \begin{equation}
    \label{small def main theorem x}
      \lim_{h\to\infty}\max\Big\{\frac1{\l_h^2},\frac{\Lambda_h}{\lambda_h}\Big\}\,\eta_{F_h}(\Om_h)+\de_{F_h}(\Om_h)=0\,.
  \end{equation}
  We want to prove the existence of an open set $\Om$, which is the disjoint union of at most $L$-many translations of $K_F$, such that, up to translations and up to extracting subsequences,
  \begin{equation}
    \label{convergence main theorem x}
     \lim_{h\to\infty}\big|\F_h(\Om_h)-\F(\Om)\big|+|\Om_h\Delta\Om|=0\,.
  \end{equation}
  {color{blue}We now turn to the proof of \eqref{convergence main theorem x}.

  \medskip

  \noindent {\it Step two}: Next}  {color{ForestGreen}

  \medskip

  \noindent {\it Step two}: We now turn to the proof of \eqref{convergence main theorem x}. Next} we obtain some immediate compactness properties for the sets $\Om_h$. Since \eqref{hp main theorem xx} implies $\F_h(E) \ge m P(E)$ for any set of finite perimeter $E$, \eqref{bounds x} implies
  \begin{equation}
    \label{bounded perimeter and volume}
    \sup_{h\in\N}|\Om_h|+P(\Om_h)<\infty\,.
  \end{equation}
  By \eqref{bounds x}, up to translations, we may assume $\Om_h\subset B_R$ for some $R>0$, and since $\sup_{h\in\N}P(\Om_h)<\infty$, the compactness theorem for sets of finite perimeter implies that, up to subsequences, there exists a bounded set of finite perimeter $\Om$ such that $|\Om_h\Delta\Om|\to 0$ as $h\to\infty$.

  \medskip

  \noindent {\it Step three}: We show that, if $u_h:\R^{n+1}\to(-\infty,0]$ denotes the $F_h$-torsion potential of $\Om_h$ extended by zero outside of $\Om_h$, then, up to extracting further subsequences,
  \[
  \mbox{$u_h\to u$ uniformly on $\R^{n+1}$}
  \]
  where $|\{u<0\}\setminus\Om|=0$ and
  \begin{equation}
    \label{formula x}
  u(x)=\sum_{j\in J}\min\Big\{\frac{F_*(x-x_j)^2-s_j^2}{2(n+1)},0\Big\}\qquad x\in\R^{n+1}\,,
  \end{equation}
  where $J$ is an at most countable set, $x_j\in\R^{n+1}$ and $s_j>0$. Indeed, by \eqref{estimate 1}, \eqref{estimate 2} and \eqref{small def main theorem x} we have that
  \begin{equation}
      \label{estimate 1 x}
      \lim_{h\to\infty}\int_{\Om_h}\Big|\nabla(\nabla_{F_h}u_h)-\frac{\Id}{n+1}\Big|^2=0\,,
  \end{equation}
  \begin{equation} \label{estimate 2 x}
  \lim_{h\to\infty}\int_{\pa\Om_h}\Big|\frac{1}{n+1}-F_h(\nabla u_h)\Big|^2\,F_h(\nu_{\Om_h})=0\,,
  \end{equation}
  \begin{equation}
      \label{estimate 1b recall}
      \lim_{h\to\infty}\int_{\Om_h}\Big|\na^2u_h-\frac{\na^2((F_h)_*/2)^2(\na {F_h}(\na u_h))}{n+1}\Big|^2=0\,.
  \end{equation}
  By the universal Lipschitz estimate of Proposition \ref{proposition lipschitz estimate} and by \eqref{mc 123},
  \begin{equation}\label{eqn: lip bd}
  \Lip(u_h)\le \frac1{m\,H^{F_h}_{\Om_h,\inf}}\le C(n,m,\k)\,.
  \end{equation}
  Since $\{u_h>0\}=\Om_h\subset B_R$, it follows that $u_h\to u$ uniformly on $\R^{n+1}$ for a non-positive Lipschitz function $u$. We notice that
  \begin{equation}
    \label{inclusion x}
    \big|\{u<0\}\setminus\Om\big|=0\,.
  \end{equation}
  Indeed, by uniform convergence, for every $\e>0$ we have $\{u<-\e\}\subset\Om_h$ if $h$ is large enough. This last fact, combined with \eqref{estimate 1 x}, implies that $\{\nabla(\nabla_{F_h}u_h)\}_{h\in\N}$ is bounded in $L^2(\{u<0\})$. In addition, $|\nabla_{F_h}u_h|\le C$ on $\R^{n+1}$ thanks to  \eqref{mm2}, \eqref{hp main theorem xx}, and \eqref{eqn: lip bd}. Since $\{u<-\e\}$ is bounded and $\e>0$ is arbitrary, we find that
  \begin{equation}
    \label{strong to v}
      \nabla_{F_h}u_h\to v\qquad\mbox{in $L^2(\{u<0\})$}
  \end{equation}
  for some $v\in H^1(\{u<0\};\R^{n+1})$. Now, $f(M)=|M-\Id|^2$ is a convex function on $\R^{n+1}\otimes\R^{n+1}$, so that
  \[
  \int_A\,f(\nabla w)\le\liminf_{h\to\infty}\int_A\,f(\nabla w_h)
  \]
  whenever $A\subset\R^{n+1}$ is open and $w_h\to w$ in $L^1(A;\R^{n+1})$ for some $w\in W^{1,1}(A;\R^{n+1})$. Applying this to $w_h=\nabla_{F_h}u_h$, and thanks to \eqref{estimate 1 x}, we find that
  \[
  \int_{\{u<-\e\}}\Big|\nabla v-\frac{\Id}{n+1}\Big|^2=0\qquad\forall \e>0\,,
  \]
  and thus
  \[
  \nabla v=\frac{\Id}{n+1}\qquad\mbox{on $\{u<0\}$}.
  \]
  Therefore, if $\{A_j\}_{j\in J}$ are the connected components of the open set $\{u<0\}$ (here $J$ is an at most countable set), then there exists $x_j\in\R^{n+1}$ such that
  \begin{equation}
    \label{terms of}
      v(x)=\frac{x-x_j}{n+1}\qquad\forall x\in A_j\,.
  \end{equation}
  We now need to translate \eqref{terms of} in terms of $u$. To this end, we claim that
  \begin{equation}
    \label{eq:vinsubdifferential}
    v(x)\in\partial (F^2/2)(\nabla u(x))\qquad\mbox{for a.e. $x\in\R^{n+1}$}\,.
  \end{equation}
  Indeed, by the convexity of $F_h^2$ and since $2\,\nabla_{F_{h}}u_h(x)=\nabla (F_h^2)(\nabla u_h(x))$, we have that
  \begin{equation}\label{eq:subdifferentialatFh}
  F_h(z)^2\ge F_h(\nabla u_h(x))^2+2\,\nabla_{F_{h}}u_h(x)\cdot\big(z-\nabla u_h(x)\big)\,\qquad\forall z\in\R^{n+1}\,.
  \end{equation}
  By \eqref{convergence hp}, \eqref{strong to v} and the uniform Lipschitz bound on the potentials $u_h$, \eqref{eq:subdifferentialatFh} implies that, for almost every $x\in\R^{n+1}$ and for every $z\in\R^{n+1}$,
  \begin{equation}\label{eq:subdifferentialatF}
  F(z)^2\ge F(\nabla u_h(x))^2+2\,v(x)\cdot\big(z-\nabla u_h(x)\big)+o(1)\qquad\mbox{as $h\to\infty$}
  \end{equation}
  where $o(1)\to 0$ as $h\to\infty$. Handling $\nabla u_h$ is more delicate, because we only have $\nabla u_h\rightharpoonup\nabla u$ in $L^2(\R^{n+1})$. However, by Mazur's lemma, for every $h\in\N$ there exist $N_h\ge h$ and $\{\alpha^h_l\}_{l=h}^{N_h}\subset[0,1]$ with $\sum_{l=h}^{N_h}\alpha^h_{l}=1$ such that
  \begin{equation}
    \label{mazur}
      \sum_{l=h}^{N_h}\alpha^h_{l}\nabla u_{l}\to\nabla u\qquad\mbox{in $L^2(\R^{n+1})$ and a.e. on $\R^{n+1}$}\,.
  \end{equation}
  In particular,  thanks to \eqref{mazur}, this implies that by convexity of $F^2$ and by \eqref{eq:subdifferentialatF}, for almost every $x\in\R^{n+1}$ and for every $z\in\R^{n+1}$ we have
  \begin{align*}
    F(z)^2=\sum_{l=h}^{N_h}\alpha^h_{l} F(z)^2&\ge \sum_{l=h}^{N_h}\alpha^h_{l}F(\nabla u_h(x))^2+2\sum_{l=h}^{N_h}\alpha^h_{l}\,v(x)\cdot(z-\nabla u_h(x))+{\rm O}(1)\\
    &\ge  F\Big(\sum_{l=h}^{N_h}\alpha^h_{l}\nabla u_h(x)\Big)^2+2\, v(x)\cdot\Big(z-\sum_{l=h}^{N_h}\alpha^h_{l}\nabla u_h(x)\Big)+{\rm O}(1)
    \\
    &= F(\nabla u(x))^2+2\,v(x)\cdot(z-\nabla u(x))\,.
  \end{align*}
  This proves \eqref{eq:vinsubdifferential}. Inverting \eqref{eq:vinsubdifferential} using \eqref{eq:invertingsubdifferentials}, we have that
  $$
  \nabla u(x)\in\partial (F_*^2/2)(v(x)) \qquad\mbox{for a.e. $x\in\R^{n+1}$}\,.
  $$
  By \eqref{terms of}, and since $F_*$ is differentiable almost everywhere on $\R^{n+1}$, this implies that
  $$
  \nabla u(x)=\nabla(F_*^2/2)(x-x_j)\qquad\mbox{for a.e. $x\in A_j$}\,.
  $$
  We thus conclude that there exist $c_j\in \R$ such that
  $$
  u(x)=\frac{F_*(x-x_j)^2}{2(n+1)}+c_j\qquad\mbox{for a.e. $x\in A_j$.}
  $$
  Since $u<0$ in $A_j$ and $u=0$ on $\pa A_j$, we deduce that $c_j>0$ and $A_j=x_j+s_j\,K_F$. Hence \eqref{formula x} follows.

  \bigskip

  \noindent {\it Step four}: The next few steps of the argument are devoted to showing that $s_j=1$ for every $j\in J$. To begin, we prove that
  \begin{equation}\label{eqn: conv of sq}
 \lim_{h\to \infty} \int_{\Om_h}F_h(\nabla u_h)^2=\int_{\{u<0\}}F(\na u)^2 \,.
  \end{equation}
  First note that by the convergence of $F_h$ to $F$ \eqref{convergence hp} and the Lipschitz bound on $u_h$ \eqref{eqn: lip bd}, we have
  \begin{equation}\label{eqn: killed an h}
  \int_{\Om_h}F_h(\nabla u_h)^2=\int_{\Om_h}F(\na u_h)^2+o(1)\qquad\mbox{as $h\to\infty$}
  \end{equation}
  where $o(1)$ denotes a vanishing sequence. Since $\{u<-\e\}\subset\Om_h$ for $h$ large enough, by $\na u_h \rightharpoonup \na u_h$ in $L^2$ and the convexity of $F$, for every $\e>0$ we have
  \[
  \liminf\limits_{h\rightarrow \infty}\int_{\Om_h}F(\na u_h)^2\ge \liminf\limits_{h\rightarrow \infty}\int_{\{u<-\e\}}F(\na u_h)^2\ge\int_{\{u<-\e\}}  F(\na u)^2\,.
  \]
  Letting $\e\rightarrow 0$ we see that, by \eqref{eqn: killed an h},
 \begin{equation}\label{eqn:upper energy}
  \liminf\limits_{h\rightarrow \infty}\int_{\Om_h}F_h(\na u_h)^2\ge\int_{\{u<0\}}F(\na u)^2\,.
  \end{equation}
  To prove the opposite inequality, we recall that $u_h$ is the unique minimizer of $\int_\Om F_h(\na v)^2$ among $v\in W^{1,2}_0(\Om_h)$. Setting
  \[
  \e_h=-\inf_{\R^{n+1}\setminus\Om_h}u\,,\qquad v_h=(u+\e_h)_{-}\in W^{1,2}_0(\Om_h)\,,
  \]
  we find that
  \begin{equation}\label{eqn:energy limsup}
  \int_{\Om_h}F(\na u_h)^2\le\int_{\Om_h}F_h(\na v_h)^2=\int_{\Om_h}F(\na v_h)^2+o(1)=\int_{\{u<-\e_h\}}F(\na u)^2+o(1)
  \end{equation}
  where in the first identity we have used the convergence of $F_h$ to $F$ \eqref{convergence hp}. Combining \eqref{eqn:energy limsup} with \eqref{eqn:upper energy}, we prove \eqref{eqn: conv of sq}.

  \bigskip

  \noindent {\it Step five}: We show that
  \begin{equation}\label{eq:convergence of Fhalpha}
  \lim_{h \to \infty}\int_{\Om_h } F_h(\na u_h)^\alpha = \int_{\{u<0\} } F(\na u)^\alpha\qquad\forall\a>0\,.
  \end{equation}
  Notice that this would be obvious if we had the strong convergence of $\nabla u_h$ to $\nabla u$, and that in that case,  one could actually assert \eqref{eq:convergence of Fhalpha} with any non-negative locally bounded function $H$ replacing $F^\a$. However, we only have that $\nabla u_h\overset{*}{\rightharpoonup} \nabla u$ in $L^\infty(\R^{n+1})$. To obtain \eqref{eq:convergence of Fhalpha}, we will exploit the strict convexity of $t\in(-1,\infty)\mapsto F((1+t)\nu)^2$ through the theory of Young measures.

  Let us recall that, by the uniform Lipschitz bound \eqref{eqn: lip bd} and since $\spt\,u_h\subset B_R$ for $R$ independent of $h$, we can apply the fundamental theorem of Young measures \cite[Chapter 1, Theorem 11]{EvansWeakConv} to find a measurable family of probability measures $\{\nu_x\}_{x\in \R^{n+1}}$ such that
  \[
  \int_{\R^{n+1}}\vphi(x)\,\psi(\nabla u_h(x))\,dx\to\int_{\R^{n+1}}\vphi(x)\int_{\R^{n+1}}\psi(\xi)\,d\nu_x(\xi)\qquad\forall \vphi\in L^1(B_R)\,,\psi\in C^0_c(\R^{n+1})\,.
  \]
  In particular, we easily deduce that
  \begin{equation}
    \label{average}
      \nabla u(x)=\int_{\R^{n+1}}\xi\,d\nu_x(\xi)\qquad\mbox{for a.e. $x\in\R^{n+1}$}\,.
  \end{equation}
  By \eqref{eqn: conv of sq} and by convexity of $F^2$, if $p$ denotes a measurable selection of $\pa F^2(\nabla u)$ on $\R^{n+1}$ (that is, $p:\R^{n+1}\to\R^{n+1}$ is a measurable map with $p(x)\in\pa F^2(\nabla u(x))$ for almost every $x\in \R^{n+1}$), then
  \begin{eqnarray*}
  \int_{\R^{n+1}}F(\na u)^2&=&\lim_{h\to \infty} \int_{\R^{n+1}} F(\na u_h)^2 \, dx = \int_{B_R}dx \int_{\R^{n+1}} F(\xi)^2 \, d\nu_x(\xi)
  \\
  & \geq& \int_{B_R}dx \int_{\R^{n+1}} F(\na u(x))^2 + p(x) \cdot (\xi -\na u(x) ) \, d \nu_x(\xi)
  \\
  &=&\int_{\R^{n+1}}F(\nabla u)^2+\int_{B_R}\,dx\int_{\R^{n+1}}p(x) \cdot (\xi -\na u(x) ) \, d \nu_x(\xi)=\int_{\R^{n+1}}F(\nabla u)^2\,,
  \end{eqnarray*}
  where in the last identity we have used \eqref{average}. Hence, for every measurable selection $p$ of $\pa F^2(\na u)$,  almost every $x\in\R^{n+1}$, and $\nu_x$-almost every $\xi\in\R^{n+1}$, we have
  \[
  F(\xi)^2 = F(\na u(x))^2 + p(x)\cdot (\na u(x) -\xi)\qquad\mbox{for a.e. $x\in\R^{n+1}$ and }\,.
  \]
  As $t\mapsto F(\nu+t\,z)^2$ is strictly convex and increasing on $t>0$ whenever $z$ is an outer normal direction to $\{F<F(\nu)\}$ at $\nu$, this in turn implies that for almost every $x\in\R^{n+1}$
  \[
  \spt(\nu_x) \subset \{F= F(\na u(x))\} \,.
  \]
  Consequently, for any $g\in C^0(\R^{n+1})$ with $g(0)=0$, we have
  \[
  \lim_{h\to \infty} \int_{\R^{n+1}} g(\na u_h )\,dx = \int_{\R^{n+1}}\,dx \int_{\R^{n+1}} g(\xi) \, d\nu_x(\xi)=\int_{B_R} g(\na u)\,.
  \]
  Recalling that
  \[
  \int_{B_R} F_h(\na u_h)^\alpha = \int_{B_R} F(\na u_h)^\alpha + o(1)\qquad\mbox{as $h\to\infty$}
  \]
  and setting $g=F^\a$, the claim follows.

  \bigskip

  \noindent {\it Step six}: We prove that if $\a>0$, then
  \begin{equation}\label{eq:pohozaevalpha}
  \frac{n+1+\alpha}{(n+1)^2}\int_{\R^{n+1}} F(\nabla u)^\alpha=\frac{n+2+\alpha}{n+1}\int_{\R^{n+1}} F(\nabla u)^{\alpha+1}.
  \end{equation}
  Indeed, let us consider the vector field
  \begin{equation*}
  \theta_{\a,h}=F_h(\nabla u_h)^\alpha\,\nabla_{F_h}u_h\,.
  \end{equation*}
Recalling that $\nabla_{F_h}u_h=F_h(\na u_h)\na F_h(\na u_h)$, we have
\begin{eqnarray*}
\nabla \big(F_h(\nabla u_h)^{\alpha+1}\big)\cdot\big(\nabla_{F_h}u_h\big)
=
(\alpha+1)F_h(\nabla u_h)^{\alpha+1} \na^2 u_h[\nabla F_h(\nabla u_h),\nabla F_h(\nabla u_h)]
\end{eqnarray*}
so that $\Div(\nabla_{F_h}u_h)=1$ yields
\begin{eqnarray*}
\Div \theta_{\a+1,h}&=&F_h(\nabla u_h)^{\alpha+1}+(\alpha+1)F_h(\nabla u_h)^{\alpha+1} \na^2 u_h[\nabla F_h(\nabla u_h),\nabla F_h(\nabla u_h)]
\\
&=& \bigg(1+\frac{\alpha+1}{n+1}\bigg)F_h(\nabla u_h)^{\alpha+1}
\\
&&+ (\alpha+1)\, F_h(\nabla u_h)^{\alpha+1}\,
\bigg(\nabla^2u_h-\frac{\na^2[(F_h)_*/2]^2(\na {F_h}(\na u_h))}{n+1}\bigg)[\nabla F_h(\nabla u_h),\nabla F_h(\nabla u_h)],
\end{eqnarray*}
where we have used the identity
\[
\na^2[H^2/2](\nu)[\nu,\nu]=1\qquad\mbox{with $H=(F_h)_*$ and $\nu=\nabla F_h(\nabla u_h)$}\,.
\]
By \eqref{estimate 1b recall}, and taking into account the Lipschitz bound \eqref{eqn: lip bd}, we find
\begin{equation}\label{eq:divofthetaalpha+1}
\int_{\Omega_h}\Div \theta_{\a+1,h}=\frac{n+2+\alpha}{n+1}\int_{\Om_h} F_h(\nabla u_h)^{\alpha+1}+o(1)\qquad\mbox{as $h\to\infty$}\,.
\end{equation}
At the same time, by the divergence theorem, the Lipschitz bound, and \eqref{estimate 2 x}, we find
\begin{eqnarray*}
\int_{\Omega_h}\Div\, \theta_{\a+1,h}
&=&
\int_{\partial\Omega_h} F_h(\nabla u)^{\alpha+1} \nabla_{F_h}u\cdot\nu_{\Om_h} \\\nonumber
&=&
\frac{1}{n+1}\int_{\partial\Omega_h} F_h(\nabla u)^{\alpha} \nabla_{F_h}u\cdot\nu_{\Om_h}
\\
&&+\int_{\partial\Omega_h}\left(F_h(\nabla u)-\frac{1}{n+1}\right) F_h(\nabla u)^{\alpha} \nabla_{F_h}u\cdot\nu_{\Om_h}
\\
\nonumber
&=&\frac{1}{n+1}\int_{\partial\Omega_h} F_h(\nabla u)^{\alpha} \nabla_{F_h}u\cdot\nu_{\Om_h}+o(1)
\\\nonumber
&=&\frac{1}{n+1}\int_{\Omega_h}\Div\,\theta_{\a,h}+o(1)\,,\qquad\mbox{as $h\to\infty$}\,.
\end{eqnarray*}
By using \eqref{eq:divofthetaalpha+1} with $\a$ and $\a+1$ we thus conclude that
\begin{equation}
  \label{previous}
  \frac{n+2+\alpha}{n+1}\int_{\Om_h} F_h(\nabla u_h)^{\alpha+1}=
\frac{n+1+\alpha}{(n+1)^2}\int_{\Om_h} F_h(\nabla u_h)^{\alpha}+o(1)\qquad\mbox{as $h\to\infty$}
\end{equation}
which implies \eqref{eq:pohozaevalpha} thanks to \eqref{eq:convergence of Fhalpha} if $\a>0$. Notice that \eqref{previous} holds also when $\a=0$ and that, in this case, it implies
\begin{equation}
  \label{eq:forvolume}
  \frac{n+2}{n+1}\int_{\R^{n+1}} F(\nabla u)=\frac{|\Om|}{n+1}\ge \frac{|\{u<0\}|}{n+1}\,,
\end{equation}
where the limit on the left-hand side is computed by \eqref{eq:convergence of Fhalpha}, the limit  on the right-hand side follows by $|\Om_h\Delta\Om|\to 0$, and the inequality is a consequence of \eqref{inclusion x}.

\bigskip

\noindent {\it Step seven}: We finally complete the proof. We recall from \eqref{formula x} that $\{u<0\}$ is decomposed in at most countably many open connected components $A_j=x_j+s_j\,K_F$, $j\in J$, with $u(x)=(F_*(x-x_j)^2-s_j^2)/2(n+1)$ for $x\in A_j$. Hence, by \eqref{F and its gauge},
\[
F(\nabla u(x))=\frac{F(F_*(x-x_j)\,\nabla F_*(x-x_j))}{n+1}=\frac{F_*(x-x_j)}{n+1}\qquad \forall x\in A_j\,,
\]
and by scaling, we have that
\begin{equation}\label{eq:exact}
\int_{\R^{n+1}} F(\nabla u)^\alpha=\frac{1}{(n+1)^{\alpha}}\sum_{j\in J}s_j^{n+1+\alpha}\int_{K_F} F_*^\alpha\,.
\end{equation}
Moreover, thanks to $K_F=\{F_*<1\}$, the coarea formula, and the zero-homogeneity of $\nabla F_*$,
\[
\int_{K_F} F_*^\alpha=\int_0^1\,t^\a\,dt\int_{\{F_*=t\}}\frac{d\H^n}{|\nabla F_*|}
=\int_0^1\,t^{n+\a}\,dt\int_{\pa K_F}\frac{d\H^n}{|\nabla F_*|}=\frac{C_K}{n+1+\a}
\]
provided we set
\[
C_K=\int_{\partial K_F}\frac{d\H^n}{|\nabla F_*|}\,.
\]
Hence, \eqref{eq:exact} becomes
\begin{equation}\label{eq:exact2}
\int_{\R^{n+1}} F(\nabla u)^\alpha = \frac{C_K}{(n+1)^\alpha (n+1+\alpha)}\sum_{j\in J}s_j^{n+1+\alpha}\,.
\end{equation}
Combining  \eqref{eq:pohozaevalpha} with \eqref{eq:exact2}, we find that
\begin{eqnarray*}
\frac{C_K}{(n+1)^{\alpha+2}}\sum_{j\in J}s_j^{n+1+\alpha} =\frac{n+1+\alpha}{(n+1)^2}\int_{\R^{n+1}} F(\nabla u)^\alpha &=&
\frac{n+2+\alpha}{n+1}\int_{\R^{n+1}} F(\nabla u)^{\alpha+1}\\
&=&\frac{C_K}{(n+1)^{\alpha+2}}\sum_{j\in J}s_j^{n+2+\alpha}\,,
\end{eqnarray*}
that is
\[
\sum_{j\in J}s_j^{n+1+\alpha}=\sum_{j\in J}s_j^{n+2+\alpha}\qquad\forall \a>0\,.
\]
In particular, by the arbitrariness of $\a$,
\begin{equation*}
\sum_{j\in J}s_j^{n+3+\alpha}=\sum_{j\in J}s_j^{n+2+\alpha}=\sum_{j\in J}s_j^{n+1+\alpha}\,,\qquad\forall\a>0\,,
\end{equation*}
and thus
\begin{equation*}
\sum_{j\in J}s_j^{n+1+\alpha}(s_j-1)^2=\sum_{j\in J}s_j^{n+3+\alpha}+s_j^{n+1+\alpha}-2\,s_j^{n+2+\alpha}=0\,.
\end{equation*}
This implies that $s_j=1$ for every $j\in J$. Using \eqref{eq:forvolume}, we get
\begin{equation}
\frac{C_K}{(n+1)^2}\sum_{j\in J}s_j^{n+2}=\frac{n+2}{n+1}\int_{\R^{n+1}} F(\nabla u) =\frac{|\Omega|}{n+1}\ge \frac{|\{u<0\}|}{n+1}=\frac{C_K}{(n+1)^2}\sum_{j\in J}s_j^{n+1}\,,
\end{equation}
which, combined with $s_j=1$ for every $j\in J$, implies the finiteness of $J$ as well as $|\Omega|=|\{u<0\}|$. In particular, \eqref{inclusion x} implies
  \[
  \Om=\{u<0\}
  \]
  up to modifying $\Om$ in a set of Lebesgue measure zero, which proves the conclusion that $|\Om_h\Delta\Om|\to 0$ for an open set $\Om$ consisting of a union of finitely many disjoint translations of $K_F$. In turn, from this last property, we have $(n+1)|\Om|=\F(\Om)$, so $(n+1)|\Om_h|=\F_h(\Om_h)$ implies $\F_h(\Om_h)\to\F(\Om)$ as $h\to\infty$. Finally, $\F_h(\Om_h)\le (L+\s)\,\F_h(K_{F_h})$ gives
  \[
  \#\,J\,\F(K_F)=\F(\Om)\le (L+\s)\F(K_F)
  \]
  so $\#\,J\le L$. This completes the proof of Theorem \ref{thm main}.

  \section{Proof of Theorem \ref{thm locmin}}\label{section locmin}

  We start by recalling our assumptions. We consider a smooth elliptic integrand $F$ and a smooth potential $g$. We let $M>0$ and consider an open connected set with smooth boundary $\Om$ such that
  \begin{equation}
    \label{hp locmin}
      \frac{\F(\Om)^{n+1}}{|\Om|^n}\le M\,,\qquad\frac{\diam(\Om)}{|\Om|^{1/(n+1)}}\le M\,,\qquad \Om\subset B_M\,.
  \end{equation}
  In case (i) we assume that $\Om$ is a volume-constrained critical point of $\E$, so that, by smoothness and by the area formula, there exists a constant $\ell$ such that
  \begin{equation}
    \label{critical point of E}
    H^F_\Om+g=\ell\qquad\mbox{on $\pa\Om$}\,.
  \end{equation}
  Then a first variation argument allows one to compute that
  \begin{equation}
      \label{formula for lambda}
          \ell=H^{F,0}_\Om+\frac1{(n+1)|\Om|}\int_\Om\,\Div(g(x)\,x)\,dx\,,
  \end{equation}
  see e.g. \cite[Appendix A.1]{FigalliMaggiARMA}. Let us now set
  \[
  \Om^*=t\,\Om\,,\qquad t=\frac{H^{F,0}_\Om}n=\frac{\F(\Om)}{(n+1)|\Om|}\,.
  \]
  By \eqref{hp locmin} we easily find
  \[
  H^{F,0}_{\Om^*}=n\,,\qquad \F(\Om^*)\le M\,,\qquad \diam(\Om^*)\le M^{(n+2)/(n+1)}\,.
  \]
  By the Wulff inequality \eqref{wulff inequality},
  \begin{equation}
    \label{via 3}
      \frac1{t}=\frac{(n+1)|\Om|}{\F(\Om)}\le C(n,F)\,|\Om|^{1/(n+1)}
  \end{equation}
  so that $\Om\subset B_M$, \eqref{critical point of E} and \eqref{formula for lambda} imply
  \begin{equation}
  \label{deficit critical}
  \Big\|\frac{H^F_\Om}{H^{F,0}_\Om}-1\Big\|_{C^0(\pa\Om)}\le C(n,F)\,\Big(\|g\|_{C^0(B_M)}+M\,\|\nabla g\|_{C^0(B_M)}\Big)\,|\Om|^{1/(n+1)}\,,
  \end{equation}
  and thus
  \[
  \de_F(\Om^*)=\de_F(\Om)\le C(n,F,M,g)\,|\Om|^{1/(n+1)}\,.
  \]
  By Theorem \ref{thm main1}, for every $\e>0$ there exists $v_\e$ depending on $n$, $F$, $g$, $M$, and $\e$, such that if $|\Om|\le v_\e$, then
  \begin{equation}
    \label{vicino L}
      \Big|\Om^*\Delta \bigcup_{i=1}^L(x_i+K_F)\Big|\le\e\,,
  \end{equation}
  for some $L\ge 1$ (bounded from above in terms of $F$ and $M$) and $\{x_i\}_{i=1}^L\subset\R^{n+1}$ such that the sets $\{x_i+K_F\}_{i=1}^L$ are mutually disjoint. This proves statement (i).

  We now assume that $\Om$ is a volume-constrained $r_0$-local minimizer of $\E$, that is
  \begin{equation}
    \label{locminproof 1}
      \E(\Om)\le\E(A)\qquad\mbox{whenever $|A|=|\Om|$ and $\Om\Delta A\cc I_{r_0\,|\Om|^{1/(n+1)}}(\pa\Om)$}\,.
  \end{equation}
  We shall apply statement (i) with a choice of $\e$ depending on $r_0$, $n$, $F$, $g$ and $M$, and then we shall assume $|\Om|\le v_0$ for a suitable $v_0$ depending on $r_0$, $n$, $F$, $g$ and $M$. We now divide the argument into steps.

  \bigskip

  \noindent {\it Step one}: We claim that there exist constants $\rho_0$ (depending on $n$, $F$, $g$, $M$ and $r_0$) and $S$ (depending on $n$, $F$, and $M$) such that $\Om^*$ is a $(S,\rho_0)$-minimizer of $\F$ in $\R^{n+1}$, that is,
  \begin{equation}
    \label{rho0S minimizer}
      \F(\Om^*)\le\F(G)+S\,|G\Delta\Om^*|\qquad\mbox{whenever $\diam(G\Delta\Om^*)<\rho_0$}\,.
  \end{equation}
  (This will be used in the next step to get density estimates, see \eqref{volume density estimates} below.) In proving \eqref{rho0S minimizer} we can assume without loss of generality that
  \begin{eqnarray}
    \label{on G}
    \F(G)\le\F(\Om^*)\,,\qquad G\Delta\Om^*\cc B_{\rho_0}(x_0)\qquad\mbox{for some $x_0\in\pa\Om^*$}\,.
  \end{eqnarray}
  We first show that, if we let
  \[
  \mu=\Big(\frac{|\Om^*|}{|G|}\Big)^{1/(n+1)}
  \]
  then we have
  \begin{equation}
    \label{mu minus 1}
      |\mu-1|\le C(n)\,\frac{|\Om^*\Delta G|}{|\Om^*|}\qquad\mbox{and}\qquad|\mu-1|\le C(n,F)\,\rho_0^{n+1}\,.
  \end{equation}
  Indeed, by \eqref{vicino L} and provided $\e\le |K_F|/2$, we have
  \begin{equation}
    \label{volume Omega star}
      \Big(L-\frac{1}2\Big)|K_F|\le |\Om^*|\le (L+1)\,|K_F|\,.
  \end{equation}
 So, if $\rho_0$ is small enough to have $|B_{\rho_0}|\le |K_F|/4$, then \eqref{on G} implies
  \[
  |G|\ge|\Om^*|-|B_{\rho_0}|\ge \Big(L-\frac34\Big)\,|K_F|\ge \frac{|\Om^*|}{8}\,,\qquad\mbox{and}\qquad |G|\le |\Om^*|+\frac{|K_F|}4\le \frac32\,|\Om^*|\,,
  \]
  hence
  \[
  |\mu-1|\le C(n)\,\frac{|\Om^*\Delta G|}{|G|}\le C(n)\,\frac{|\Om^*\Delta G|}{|\Om^*|}\,,
  \]
  that is, the first estimate in \eqref{mu minus 1}. The second estimate immediately follows by combining the first one with \eqref{volume Omega star} and $|\Om^*\Delta G|\le|B_{\rho_0}|$. Now we set, for $x_0\in\pa\Om^*$ as in \eqref{on G},
  \[
  A^*=\mu\,(G-x_0)+x_0\,,\qquad A=\frac{A^*}t\,,
  \]
  and claim that $A$ is admissible in \eqref{locminproof 1}. By definition of $\mu$ we have $|A^*|=|\Om^*|$, so that $\Om^*=t\,\Om$ implies $|A|=|\Om|$. We thus need to check that
  \[
  \Om\Delta A\cc I_{r_0\,|\Om|^{1/(n+1)}}(\pa\Om)\,.
  \]
  We first claim that
  \begin{equation}
    \label{via}
    \Om^*\Delta A^*\cc I_{2\,\rho_0}(\pa\Om^*)\,.
  \end{equation}
  The argument is entirely elementary, but we include it for the sake of clarity. Let us first pick $y\in A^*\setminus\Om^*$ and let $z\in G$ be such that $y=\mu(z-x_0)+x_0$. If $z\in B_{\rho_0}(x_0)$, then
  \[
  \dist(y,\pa\Om^*)\le |y-x_0|\le \mu\,|z-x_0|\le \mu\,\rho_0\le 2\rho_0\,,
  \]
  where we have used the second inequality in \eqref{mu minus 1} and have assumed $\rho_0$ small enough to get $\mu\le 2$. On the other hand, if $z\not\in B_{\rho_0}(x_0)$, then $G\setminus B_{\rho_0}(x_0)=\Om^*\setminus B_{\rho_0}(x_0)$ implies that $z\in\Om^*$, so that a point on the segment joining $y$ and $z$ lies on $\pa\Om^*$, and
  \[
  \dist(y,\pa\Om^*)\le |z-y|= |\mu-1|\,|z-x_0|\le\diam(\Om^*)\,|\mu-1|\le \rho_0\,,
  \]
  where we have used again the second inequality in \eqref{mu minus 1} and the smallness of $\rho_0$. This proves that
  \[
  \dist(A^*\setminus\Om^*,\pa\Om^*)\le\rho_0\,.
  \]
  Now let us pick $y\in\Om^*\setminus A^*$. If $y\not\in G$, then, by \eqref{on G}, $y\in B_{\rho_0}(x_0)$ and thus $\dist(y,\pa\Om^*)\le\rho_0$. Otherwise, $y\in G\setminus A^*$, which means that the point $z$ defined by $y=\mu(z-x_0)+x_0$ cannot lie in $G$. Thus the segment joining $z$ and $y$ meets a point in the boundary of $G$, so that, again by \eqref{mu minus 1},
  \[
  \dist(y,\pa G)\le |y-z|=|\mu-1|\,|z-x_0|=|\mu-1|\,\frac{|y-x_0|}\mu\le 2\,|\mu-1|\,|y-x_0|\le 2\, \diam(\Om^*)\,|\mu-1|\le \rho_0\,,
  \]
  provided $\rho_0$ is small enough. By $\pa G\subset I_{\rho_0}(\pa\Om^*)$ we get
  \[
  \dist(y,\pa \Om^*)\le\rho_0+\dist(y,\pa G)\,,
  \]
  and thus $\dist(\Om^*\setminus A^*,\pa\Om^*)\le2\rho_0$. The proof of \eqref{via} is complete.

  By $|\Om^*|=|A^*|$, \eqref{via}, $\Om^*=t\,\Om$ and $A^*=t\,A$ we obtain that
  \begin{equation}
    \label{via 2}
      |\Om|=|A|\qquad \Om\Delta A\cc I_{2\rho_0/t}(\pa\Om)\,.
  \end{equation}
  By \eqref{via 3}, if $\rho_0$ is small enough with respect to $r_0$, $n$, $M$ and $F$, then \eqref{via 2} implies that $A$ is admissible in \eqref{locminproof 1}. By $\E(\Om)\le\E(A)$, and assuming without loss of generality that $r_0\le M$, we deduce that
  \begin{eqnarray}\label{via 4}
    \F(\Om)\le\F(A)+\|g\|_{C^0(B_{2M})}\,|\Om\Delta A|\,.
  \end{eqnarray}
 Multiplying by $t^n$, this becomes
  \begin{eqnarray*}
  \F(\Om^*)&\le&\F(A^*)+\frac{\|g\|_{C^0(B_{2M})}}t\,|\Om^*\Delta A^*|\le \F(A^*)+C(n,F)\,\|g\|_{C^0(B_{2M})}\,v_0^{1/(n+1)}\,|\Om^*\Delta A^*|
    \\
    &\le&
    \F(A^*)+\,|\Om^*\Delta A^*|
  \end{eqnarray*}
  where we have used first \eqref{via 3}, and then the smallness of $v_0$. Now, $\F(A^*)=\mu^n\,\F(G)$ so that by the first estimate in \eqref{mu minus 1} and by $\F(\Om^*)=(n+1)\,|\Om^*|$, we obtain
  \begin{eqnarray}\nonumber
  \F(A^*)&\le&\Big(1+C(n,F,M)\,\frac{|\Om^*\Delta G|}{|\Om^*|}\Big)\F(G)
  \le\F(G)+C(n,F,M)\,\frac{\F(\Om^*)}{|\Om^*|}\,|\Om^*\Delta G|
    \\\label{via 5}
    &\le&\F(G)+C(n,F,M)\,|\Om^*\Delta G|\,.
  \end{eqnarray}
  Similarly, by $\F(G)\le\F(\Om^*)\le M$ and by \cite[Lemma 17.9]{maggiBOOK},
  \[
  |G\Delta A^*|\le C(n,M)\,|\mu-1|\,P(G)\le C(n,F,M)\,|\mu-1|\,\F(G)\le C(n,F,M)\,|\Om^*\Delta G|\,,
  \]
  so $|\Om^*\Delta A^*|\le C(n,F,M)\,|\Om^*\Delta G|$ by triangular inequality. Combining this last estimate with \eqref{via 3}, \eqref{via 4} and \eqref{via 5} we conclude that \eqref{rho0S minimizer} holds.

  \bigskip

  \noindent {\it Step two}: We show that $L=1$ in \eqref{vicino L}. Starting from \eqref{rho0S minimizer}, and assuming without loss of generality that $\rho_0\le 1/S$, a standard argument shows the existence of  $\k\in(0,1)$ depending on $F$ and $n$ only such that
  \begin{equation}
    \label{volume density estimates}
      \k\,|B_\rho(x)|\le|\Om^*\cap B_\rho(x)|\le(1-\k)\,|B_\rho(x)|\,,\qquad\forall x\in\pa\Om^*\,,\rho<\rho_0\,.
  \end{equation}
  (See, e.g., \cite[Theorem 21.11]{maggiBOOK} for the case when $\F$ is the classical perimeter and use the constants $m_F$ and $M_F$ bounding the anisotropy to adapt the proof to the anisotropic case.)

  Let us introduce the shorthand $K_0 =\bigcup_{i=1}^L(x_i+K_F)$, and let $x\in\pa\Om^*$ be such that
  \[
  h_0=\dist(x,\pa K_0)=\sup_{z\in \pa \Om^*}\dist(z,\pa K_0)\,.
  \]
 If $x \in K_0^c,$
  then by the lower density estimate in \eqref{volume density estimates} and \eqref{vicino L},
  \[
  \k\,|B_{\min\{h_0,\rho_0\}}(x)|\le|\Om^*\cap B_{\min\{h_0,\rho_0\}}(x)|\le|\Om^*\setminus K_0|\le \e\,.
  \]
  If instead $x \in K_0$, then the analogous argument using the upper density estimate in \eqref{volume density estimates} shows that $ \k\,|B_{\min\{h_0,\rho_0\}}(x)|\le \e$ in this case as well.
  Taking $\e$ small enough with respect to $\rho_0$ and $\k$, and therefore with respect to $r_0$, $n$, $F$, $g$ and $M$, we find that $h_0<\rho_0$ and in particular that
  \[
  h_0^{n+1}\le C(n,F)\,\e\,.
  \]
  Furthermore, up to possibly decreasing $v_0$, we find that
  \[
  h_1^{n+1} = \sup_{z\in\pa K_0}\dist(z,\pa\Om)\le C(n,F)\,\e\,
  \]
  as well. Indeed, otherwise, we could find $r>0$, a sequence $\{\Om_h\}$ converging to $K_0$ in $L^1$ and $x\in K_0$ such that $B_{r}(x) \cap \pa \Om_h$ is empty for $h$ sufficiently large. Clearly, $K_0$  satisfies a lower perimeter density estimate. Pairing this estimate with the ellipticity of $F$ and the lower semi-continuity of $\F$ implies that
  \[
  cr^n \leq \F(K_0; B_r(x)) \leq \liminf_{h\to\infty} \F(\Om_h; B_r(x)) = 0,
  \]
  yielding a contradiction.
   We conclude that
  \begin{equation}
    \label{hd vs vol}
    h=\hd(\pa\Om^*,\pa K_0)\le C(n,F)\,\e^{1/(n+1)}\,.
  \end{equation}

Since $\Om$ is connected, so is $\Om^*\subset\bigcup_{i=1}^LI_h(x_i+K_F)$. Hence, if $L\ge 2$, then (up to relabeling the $x_i$s) we have
   \[
  I_h(x_1+K_F)\cap I_h(x_2+K_F)\ne\emptyset\,.
  \]
In particular,  if $z_1\in x_1+\pa K_F$ and $z_2\in x_2+\pa K_F$ are such that $|z_1-z_2|=\dist(x_1+K_F,x_2+K_F)$, then $|z_1-z_2|<2h\le C(n,F)\,\e^{1/(n+1)}$. Moreover, by \eqref{hd vs vol} there exists $x\in\pa\Om^*$ such that $|z_1-x|<C(n,F)\,\e^{1/(n+1)}$. Thus there exist $w_1$ and $w_2$ such that $w_1+(x_1+K_F)$ and $w_2+(x_2+K_F)$ are disjoint, with a common boundary point at $x$  and
  \begin{equation}
    \label{w1w2}
      \max\{|w_1|,|w_2|\}\le C(n,F)\,\e^{1/(n+1)}\,.
  \end{equation}
  By the upper estimate in \eqref{volume density estimates}, we find that for every $\rho<\rho_0$
  \begin{eqnarray}\nonumber
    \k\,|B_{\rho/2}(x)|&\le&|B_{\rho/2}(x)\setminus\Om^*|\le
    \Big|B_{\rho/2}(x)\setminus K_0\Big|+\e
    \\\nonumber
    &\le&
    \Big|B_{\rho/2}(x)\setminus\big((x_1+K_F)\cup(x_2+K_F)\big)\Big|+\e
    \\\label{st}
    &\le&
    \Big|B_{\rho/2}(x)\setminus\big((w_1+x_1+K_F)\cup(w_2+x_2+K_F)\big)\Big|+C(n,F)\,\e^{1/(n+1)}\,,
  \end{eqnarray}
  where in the last step we have used \eqref{w1w2}. By definition of $w_1$ and $w_2$ and by smoothness of $K_F$,
  \[
  \lim_{\rho_0\to 0^+}\rho_0^{-1-n}\Big|B_{\rho_0/2}(x)\setminus\big((w_1+x_1+K_F)\cup(w_2+x_2+K_F)\big)\Big|=0\,.
  \]
  Hence if $\e$ is small enough with respect to $\rho_0$, \eqref{st} leads to a contradiction. This shows that $L=1$. As a consequence, $K_0= x_1 +K_F$ satisfies upper and lower volume density estimates, so arguing as above, we find that
  \begin{equation}
    \label{hd vs vol 2}
    \hd\Big(\pa\Om^*,x_1+\pa K_F\Big)\le C(n,F)\,\Big|\Om^*\Delta (x_1+K_F)\Big|^{1/(n+1)}\le C(n,F)\,\e^{1/(n+1)}\le\frac{r_0}{2\,C_*(n,F)}
  \end{equation}
  where the last inequality is obtained by further decreasing $\e$ in terms of $r_0$, $n$ and $F$, with $C_*(n,F)$ as in \eqref{via 3}, that is to say, by taking $t^{-1}\le C_*(n,F)\,|\Om|^{1/(n+1)}$. In this way, we obtain
  \begin{equation}\label{ttt}
  \hd\Big(\pa\Om,p_1+\pa(K_F/t)\Big)\le \frac{r_0}{2\,C_*(n,F)\,t}\le \frac{r_0}2\,|\Om|^{1/(n+1)}\,,
  \end{equation}
  where $p_1=x_1/t$. This fact will be used in the next step.

  \medskip

  \noindent {\it Step three}: We now prove \eqref{locmin L1 close} and \eqref{locmin hd close} by comparing $\Om^*$ with a scaling of the Wulff shape of the same volume. To this end, we introduce the scale invariant {\it Wulff deficit} of $\Om\subset\R^{n+1}$, defined by
  \[
  \de_{W}(\Om)=\frac{\F(\Om)}{(n+1)|K_F|^{1/(n+1)}|\Om|^{n/(n+1)}}-1\,.
  \]
  We first claim that
  \begin{equation}
    \label{delta eps}
    \de_W(\Om)\le C(n,F)\,|\Om^*\Delta(x_1+K_F)|\le C(n,F)\,\e\,.
  \end{equation}
  Indeed, $\F(K_F)=(n+1)|K_F|$ and $\F(\Om^*)=(n+1)|\Om^*|$ (as $H^{F,0}_{\Om^*}=n$), so that
  \[
  |\F(\Om^*)-\F(K_F)|\le (n+1)\big||\Om^*|-|K_F|\big|\le (n+1)\,|\Om^*\Delta(x_1+K_F)|\,,
  \]
  while
  \begin{eqnarray*}
  \big|\F(K_F)-(n+1)|K_F|^{1/(n+1)}|\Om^*|^{n/(n+1)}\big|&\le& C(n,F)\,\big||K_F|^{n/(n+1)}-|\Om^*|^{n/(n+1)}\big|
  \\
  &\le& C(n,F)\,\big||K_F|-|\Om^*|\big|^{n/(n+1)}
  \\
  &\le& C(n,F)\,|\Om^*\Delta(x_1+K_F)|\,,
  \end{eqnarray*}
  thanks also to \eqref{volume Omega star}.
  Again using \eqref{volume Omega star} and the scaling invariance of $\de_W(\Om)$, we deduce \eqref{delta eps}.

  This said, let $s=(|K_F|/|\Om|)^{1/(n+1)}$. By \eqref{ttt},
  \[
  \hd\Big(\pa\Om,p_1+\pa (K_F/s)\Big)\le  \frac{r_0}2\,|\Om|^{1/(n+1)}+C(n,F)\,\Big|\frac1t-\frac1s\Big|
  \]
  where by \eqref{via 3} and \eqref{delta eps} and provided we take $\e$ small enough in terms of $r_0$,
  \begin{eqnarray}\label{via 6}
    \Big|\frac1t-\frac1s\Big|&=&\frac{\de_W(\Om)}t\le C(n,F)\,|\Om|^{1/(n+1)}\,\de_W(\Om)
    \\\nonumber
    &\le&\frac{r_0}2\,|\Om|^{1/(n+1)}\,.
  \end{eqnarray}
  We have thus shown that
  \[
  \hd\Big(\pa\Om,p_1+\pa (K_F/s)\Big)\le  r_0\,|\Om|^{1/(n+1)}\,,
  \]
  and since $|p_1+(K_F/s)|=|\Om|$, this implies that $A=p_1+(K_F/s)$ is an admissible competitor in \eqref{locminproof 1}.

  We can thus obtain \eqref{locmin L1 close} by the quantitative Wulff inequality of \cite{FigalliMaggiPratelliINVENTIONES} by direct comparison with the Wulff shape, as in the case of global minimizers \cite{FigalliMaggiARMA}. We repeat the simple argument for the convenience of the reader. By \eqref{locminproof 1}, we get $\E(\Om)\le\E(K_F/s)$, which in turn implies
  \begin{eqnarray}\nonumber
  \de_W(\Om)&=&\frac{\F(\Om)}{\F(K_F/s)}-1\le \frac1{\F(K_F/s)}\,\Big|\int_\Om g-\int_{p_1+(K_F/s)}g\Big|
  \\\label{via 7}
  &\le&\frac{C(n,F)}{|\Om|^{n/(n+1)}}\,\|g\|_{C^0(B_{2M})}\,|\Om\Delta (p_1+(K_F/s))|\,.
  \end{eqnarray}
  Notice that $x_1=t\,p_1$ could have actually been chosen so to satisfy
  \[
  |\Om\Delta (p_1+(K_F/s))|=\min_{p\in\R^{n+1}}|\Om\Delta (p+(K_F/s))|\,.
  \]
  Correspondingly, by the quantitative Wulff inequality of \cite{FigalliMaggiPratelliINVENTIONES},
  \[
  \de_W(\Om)\ge c(n)\,\Big(\frac{|\Om\Delta (p_1+(K_F/s))|}{|\Om|}\Big)^2\,,
  \]
  so that
  \[
  \Big(\frac{|\Om\Delta (p_1+(K_F/s))|}{|\Om|}\Big)^2\le  C(n,F,g,M)\,\frac{|\Om\Delta (p_1+(K_F/s))|}{|\Om|}\,|\Om|^{1/(n+1)},
  \]
  and \eqref{locmin L1 close} follows. Returning to \eqref{hd vs vol 2} we have
  \begin{equation}
    \label{via 8}
      \hd\Big(\pa\Om,p_1+\pa(K_F/t)\Big)\le C(n,F)\,\Big|\Om\Delta (p_1+K_F/t)\Big|^{1/(n+1)}\,,
  \end{equation}
  while by \eqref{via 6},\eqref{via 7}, and \eqref{locmin L1 close},
  \begin{eqnarray*}
    \Big|\frac1t-\frac1s\Big|&\le& C(n,F)\,\de_W(\Om)\,|\Om|^{1/(n+1)}
  \\
  &\le&C(n,F,g,M)\,\frac{|\Om\Delta (p_1+(K_F/s))|}{|\Om|}\,|\Om|^{2/(n+1)}
  \le C(n,F,g,M)\,|\Om|^{3/(n+1)}\,
  \end{eqnarray*}
 By \eqref{via 3} and \cite[Lemma 17.9]{maggiBOOK}
  \begin{eqnarray*}
    \Big|(p_1+K_F/s)\Delta (p_1+K_F/t)\Big|\le C(F)\,\frac1{t^n}\,\Big|\frac1t-\frac1s\Big|\le C(n,F,g,M)\,|\Om|^{n/(n+1)}\,|\Om|^{3/(n+1)}\,,
  \end{eqnarray*}
  while
  \begin{eqnarray*}
  \hd\Big(p_1+\pa(K_F/t),p_1+\pa(K_F/s)\Big)\le C(F)\,\Big|\frac1t-\frac1s\Big| \le C(n,F,g,M)\,|\Om|^{3/(n+1)}\,.
  \end{eqnarray*}
  By \eqref{via 8}, we conclude
  \begin{eqnarray*}
    \frac{\hd\Big(\pa\Om,p_1+\pa(K_F/s)\Big)}{|\Om|^{1/(n+1)}}
    &\le&
    \frac{\hd\Big(\pa\Om,p_1+\pa(K_F/t)\Big)}{|\Om|^{1/(n+1)}}+\frac{\hd\Big(p_1+\pa(K_F/t),p_1+\pa(K_F/s)\Big)}{|\Om|^{1/(n+1)}}
    \\
    &\le& C(n,F)\,
    \Big(\frac{|\Om\Delta (p_1+K_F/t)|}{|\Om|}\Big)^{1/(n+1)}+ C(n,F,g,M)\,\frac{|\Om|^{3/(n+1)}}{|\Om|^{1/(n+1)}}
    \\
    &\le& C(n,F,g,M)\,|\Om|^{1/(n+1)^2}\,,
  \end{eqnarray*}
  where in the last inequality we have used \eqref{locmin L1 close}. This proves \eqref{locmin hd close}.

\bibliography{references}
\bibliographystyle{alpha}

\end{document}